\documentclass[a4paper,11pt,reqno]{amsart} 
\usepackage[utf8]{inputenc}
\usepackage{enumerate}
\usepackage{color}
\usepackage[dvipsnames]{xcolor}
\usepackage{bm,bbm}
\usepackage{cancel}
\usepackage{enumitem}
\usepackage{physics}
\usepackage{amsmath,euscript}
\usepackage{amsthm}
\usepackage{amssymb}
\usepackage{graphicx}
\usepackage{mathrsfs}
\usepackage{epsf}
\usepackage{xcolor}
\usepackage{psfrag}
\usepackage{enumerate}
\usepackage{amsfonts,amssymb,amsthm,amsmath}
\usepackage{hyperref}
\usepackage[numbers]{natbib}
\usepackage{enumitem} 
\usepackage{xspace}
\usepackage{soul}

\theoremstyle{theorem}
\newtheorem{thm}{Theorem}[section]

\newtheorem{lem}[thm]{Lemma}

\theoremstyle{definition}

\theoremstyle{remark}
\usepackage{appendix}

\numberwithin{equation}{section}

\theoremstyle{remark}

\theoremstyle{definition}

\numberwithin{thm}{section}
\numberwithin{equation}{section}

\setstcolor{red}

\newcommand{\inte}{\int_{\mathbb{R}^3}}
\newcommand{\R}{\mathbb{R}}
\newcommand{\C}{\mathbb{C}}

\newcommand{\cE}{\mathcal{E}}

\newcommand{\cK}{\mathcal{K}}

\newcommand{\cP}{\mathcal{P}}         

\newcommand{\cR}{\mathcal{R}}

\newcommand{\cX}{\mathcal{X}}

\newcommand{\fS}{\mathfrak{S}}

\newcommand{\one}{1}

\renewcommand{\Re}{\mathrm{Re}} 
\renewcommand{\Im}{\mathrm{Im}} 

\newcommand{\Ex}{X}
\newcommand{\DETAILS}[1]{}

\newcommand{\nn}{\nonumber}

\allowdisplaybreaks

\title[Mass-Critical Neutron Stars in HF and HFB theories]{Mass-Critical Neutron Stars in the Hartree--Fock and Hartree--Fock--Bogoliubov Theories}

\author[B. Chen]{Bin Chen}
\address{(Bin Chen) School of Mathematics and Statistics,  Key Laboratory of Nonlinear Analysis $\&$ Applications (Ministry of Education), Central China Normal University, Wuhan 430079, P. R. China; and Academy of Mathematics and Systems Science, Chinese Academy of Sciences, Beijing 100190, P. R. China
}
\email{binchen@amss.ac.cn}

\author[Y. Guo]{Yujin Guo}
\address{ (Yujin Guo) School of Mathematics and Statistics,  Key Laboratory of Nonlinear Analysis $\&$ Applications (Ministry of Education), Central China Normal University, Wuhan 430079, P. R. China}
\email{yguo@ccnu.edu.cn}

\author[P.T. Nam]{Phan Th\`anh Nam}
\address{(Phan Th\`anh Nam) Department of Mathematics, LMU Munich, Theresienstrasse 39, 80333 Munich, Germany}
\email{nam@math.lmu.de}

\author[D.H. Ou Yang]{Dong Hao Ou Yang}
\address{(Dong Hao Ou Yang) Department of Mathematics, LMU Munich, Theresienstrasse 39, 80333 Munich, Germany}
\email{ouyang@math.lmu.de}

\begin{document}

\begin{abstract} We investigate the ground states of neutron stars and white dwarfs in the Hartree-Fock (HF) and Hartree-Fock-Bogoliubov (HFB) theories. It is known that the system is stable below a critical mass, which depends on the gravitational constant, while it  becomes unstable if the total mass exceeds the critical mass. We prove that if the total mass is at the critical mass, then the HFB minimizers do not exist for any gravitational constant, while the HF minimizers exist for every gravitational constant except for a countable set, which is fully characterized by the Gagliardo-Nirenberg inequality for orthonormal systems. Our results complement the existence results in the sub-critical mass case established in [E. Lenzmann and M. Lewin, Duke Math. J., 2010].
\end{abstract}

\date{\today}

\maketitle


\section{Introduction}\label{sec:intro}

From the first principles of quantum mechanics, the ground-state energy   for  a large system of $N$ (spinless) fermions in $\mathbb{R}^3$ is typically 
described by a linear variational problem in $L^2(\mathbb{R}^{3N})$, whose complexity however grows dramatically as $N \to 
\infty$. Therefore, in practice, the complicated many-body quantum theory is usually replaced by simplified models that are 
nonlinear but depend on a small number of variables. In this context, the {\em Hartree--Fock} (HF) and {\em 
Hartree--Fock--Bogoliubov} (HFB) theories are popular approximate models for investigating the macroscopic behavior of 
interacting fermionic systems. In the present paper, we are interested in the ground states of neutron stars and white dwarfs 
within the HF and HFB theories, focusing on their existence and nonexistence at the critical mass with respect to  the 
gravitational collapse of the system.

In the mathematical description of neutron stars and white dwarfs, the kinetic energy is described by the pseudo-relativistic 
operator $ \sqrt{-\Delta + m^2} - m$, where the constant $m>0$ denotes the rest mass of each particle,  and the interactions 
among the particles are described through the Newtonian potential  $-\kappa |x|^{-1}$, where the constant $\kappa > 0$ is 
proportional to the gravitational   constant. Typically, $\kappa>0$ is equal to the product of $m^2$ and the gravitational 
constant, but here we  treat $m>0$ and $\kappa>0$ as two independent parameters. Therefore, the Hamiltonian of the system is
\begin{equation}\label{h}
H_{N}=\sum_{j=1}^{N}\Big(\sqrt{-\Delta_{x_{j}}+m^{2}}-m\Big)-\kappa\sum_{1\leq j<k\leq N}\frac{1}{|x_{j}-x_k|}
\end{equation}
acting on the anti-symmetric space $\bigwedge^NL^{2}(\R^3,\C^q)$, where $q\ge 1$ denotes the spin of the particles.

A remarkable phenomenon in the above system is that the kinetic operator is comparable to the interaction potential,
 leading to the existence of a {\em critical mass} for gravitational collapse:  for given $m>0$ and $\kappa>0$, the system is stable, meaning that the energy is finite, if the total number of particles is strictly below a constant $N_*>0$, while it becomes unstable, meaning that the energy is $-\infty$, if the total 
number of particles exceeds $N_*$. The value of the critical mass was predicted by Chandrasekhar in 1931  (cf. 
\cite{Chandrasekhar-31}): in standard units, the {\em Chandrasekhar limit} is about 1.4 times the mass of the Sun. In 
\cite{LieYau-87}, Lieb and Yau rigorously proved that the critical mass is independent of $m$ and satisfies $N_*\sim 
(\tau_c/\kappa)^{3/2}$ as $\kappa\to 0$, where $\tau_c\approx 2.677$ is a universal constant defined from a Thomas--Fermi type 
variational problem. This improves earlier results of Lieb and Thirring \cite{LieThi-84}. We refer to the textbook 
\cite[Chapter 8]{LieSei-10} for a pedagogical discussion on the stability and instability of relativistic matter.

While the investigation in \cite{LieYau-87} focuses on comparing the leading order of the many-body ground-state energy with 
that of the Thomas-Fermi type energy, the further information about the ground states is expected to be observed in more 
refined approximations. In the HF theory, we take into account Pauli's exclusion principle for fermions by restricting the 
considerations to {\em Slater determinants}, a subclass of wave functions in $L^2(\mathbb{R}^{3N})$ that contain the least 
correlation between particles under the anti-symmetry condition. In this case, the energy functional can be described purely in 
terms of the {\em one-body density matrix} $\gamma$, $i.e.,$ a rank-$N$ projection on $L^2(\mathbb{R}^3)$. The HF energy 
functional $\cE^{\rm HF}_{m,\kappa}(\gamma)$ is hence defined by
\begin{equation}\label{functional-intro}
\begin{split}
\cE^{\rm HF}_{m,\kappa}(\gamma):=& \Tr\big((\sqrt{-\Delta+m^{2}}-m) \gamma\big) \\
&-\frac{\kappa}{2}\int_{\R^{3}}\int_{\R^{3}}\frac{\rho_{\gamma}(x)\rho_{\gamma}(y)-|\gamma(x,y)|^{2}}{|x-y|}dxdy,
\end{split}
\end{equation}
where $\gamma(x,y)$ is the integral kernel of $\gamma$ and $\rho_\gamma(x)=\gamma(x,x)$ is its one-body density. Here 
(\ref{functional-intro}) contains the full kinetic energy of the system, but its interaction energy is simplified greatly and 
contains only the {\em direct} and {\em exchange} terms. The HF approximation was originally proposed to determine the 
configurations of electrons in atoms and molecules, see,  Lieb and Simon \cite{LieSim-77} for a study in this direction, but 
later widely adopted in many different contexts.

When the attractive forces among quantum particles become significant, e.g., in gravitational physics as well as in 
superconducting materials, the HF theory admits a refined version, $i.e.,$ the HFB theory. In this framework, the one-body 
density matrix is required to satisfy $0 \leq \gamma \leq 1$ (not necessarily a projection), and allow  for the emergence of 
``Cooper pairing", which is modeled by an operator $\alpha$ on $L^2(\mathbb{R}^3)$ satisfying
\begin{align}\label{red-matrix-cond}
\begin{pmatrix}
0	&	0\\
0	&	0
\end{pmatrix}\leq \begin{pmatrix}
\gamma	&	\alpha\\
\alpha^{*}	&	1-\bar{\gamma}
\end{pmatrix}\leq\begin{pmatrix}
1	&	0\\
0	&	1
\end{pmatrix}\quad\text{on }\, L^{2}(\R^{3};\C^{q})\oplus L^{2}(\R^{3};\C^{q}).
\end{align}
The HFB energy functional of neutron stars is then given by
\begin{align}\label{HFB-energy-intro}
\cE_{m,\kappa}^{\rm HFB}(\gamma,\alpha)= \cE^{\rm HF}_{m,\kappa}(\gamma)  
-\frac{\kappa}{2}\int_{\R^{3}}\int_{\R^{3}}\frac{|\alpha(x,y)|^{2}}{|x-y|}dxdy,
\end{align}
where $\cE^{\rm HF}_{m,\kappa}(\gamma)$ and $\kappa>0$ are the same as those of (\ref{functional-intro}).
The extra {\em pairing} term given by $\alpha$  leads to various new features. The operators $(\gamma, \alpha)$ describe a 
larger class of quantum states, called {\em quasi-free states}. We refer to \cite{BacLieSol-94} for a comprehensive 
introduction to the HFB theory.

For the HF and HFB theories of neutron stars and white dwarfs, the critical mass also exists and agrees with the many-body 
description to the leading order. The existence of ground states {\em below the critical mass} was established successfully in 
the important work of Lenzmann and Lewin \cite{LenLew-10}  by using and developing the concentration-compactness method of 
Lions \cite{Lions1}. The existence problem in these effective theories, especially in the HFB theory, is conceptually 
challenging since the energy functional is {\em not} weakly lower semicontinuous, due to the attractive interaction among 
particles. Additionally, since the system is translation-invariant and its kinetic energy is nonlocal, there appear other 
technical challenges in the analysis of the system.

The aim of the present paper is to investigate the existence and nonexistence of ground states for neutron stars and white 
dwarfs at the {\em critical mass}. As noted in \cite{LieYau-87,LenLew-10}, the critical mass depends on $\kappa > 0$, but it is 
independent of $m$. It turns out that understanding the situation in the {\em massless case} $m = 0$ is of crucial importance. 
In this direction, we shall first prove the existence of ground states at the critical mass for $m = 0$ in both the HF and HFB 
theories, and then  deduce the existence or nonexistence for $m > 0$.

For the existence result with $m=0$, the key challenge  in the critical case is a lack of coercivity,	as already remarked in 
\cite[Remark 8]{LenLew-10}. In the context of the concentration-compactness method \cite{Lions1}, this requires new ideas to 
rule out the {\em dichotomy} scenarios. In the HFB theory, we shall prove the following {\em existence-nonexistence 
alternative}: the existence for the variational problem with $m=0$ implies the nonexistence for the variational problem with 
$m>0$ at the same mass. This observation follows from a simple scaling property for $m=0$, but it is useful extremely.  In 
particular, applying the existence result for the sub-critical mass case with $m>0$ (cf. \cite{LenLew-10}), we can rule out the 
dichotomy scenario at the critical mass with $m=0$, which thus leads to the existence of a minimizer.
However, as for the HF theory at the critical mass, we combine the insights from the HFB theory with an induction argument to 
handle the nonlinear constraint $\gamma = \gamma^*$. In this regard, besides the method of \cite{LenLew-10}, which follows the 
same spirit as Friesecke \cite{Friesecke-03} in the context of standard atomic Hartree-Fock and MCSCF theories, we also study a 
relaxed problem inspired by the recent work of Frank, Gontier, and Lewin \cite{ii,FGL} concerning  finite-rank Lieb--Thirring 
inequalities. As we shall see, the existence of HF ground states for $m = 0$ leads to a Gagliardo-Nirenberg-type inequality of 
orthonormal systems, which plays a key role in the variational problem for $m>0$.

The implication from the existence for $m = 0$ to the existence or nonexistence for $m > 0$ is very different in the HF and HFB 
theories. On the one hand, we find that the HFB minimizers at the critical mass do not exist for all choices of $m > 0$ and 
$\kappa > 0$. On the other hand, the HF minimizers at the critical mass always exist for every $\kappa > 0$, except for a 
countable set characterized by the many-body Gagliardo–Nirenberg inequality. We interpret the enhanced binding ability of the 
HF theory as a consequence of the strong nonlinear constraint $\gamma = \gamma^2$.




\medskip

Now we state the main results of the present paper, and explain the main ingredients of the proofs. We  start with  the HFB theory, whose mathematical setting is more demanding. We then turn to the HF theory, which is simpler but also has  independent interest.

\subsection{Hartree--Fock--Bogoliubov theory}

To formulate the HFB variational problem, we introduce the (real) Banach space of density matrices
\begin{align}\label{X}
\cX&=\big\{(\gamma,\alpha)\in\fS_{1}\times\fS_{2}:\, 
\gamma^{*}=\gamma,\,\alpha^{T}=-\alpha,\,\|(\gamma,\alpha)\|_{\cX}<\infty\big\},
\end{align}
equipped with the norm
\begin{align}\label{X-norm} 
\|(\gamma,\alpha)\|_{\cX}&:=\|(\one-\Delta)^{1/4}\gamma(\one-\Delta)^{1/4}\|_{\fS_{1}}+\|(\one-\Delta)^{1/4}\alpha\|_{\fS_{2}},
\end{align}
where we denote by $\fS_{1}$ and $\fS_{2}$ the spaces of trace-class and Hilbert-Schmidt operators on $L^{2}(\R^{3};\C)$, 
respectively.  Here we ignore the spin of the particles, as it plays no significant role in the analysis. Of course, all 
results in this paper remain valid for any fixed spin $q$.

The space $\cX$ consists of the (generalized) one-body density matrices with finite kinetic energy. Due to the fermionic 
property of the particles, we have to assume the supplementary condition \eqref{red-matrix-cond}. This leads to the following 
variational space
\begin{align*}\label{K-lamb}
\cK&:=\Big\{(\gamma,\alpha)\in\cX:\,\ \begin{pmatrix}
0	&	0\\
0	&	0
\end{pmatrix}\leq \begin{pmatrix}
\gamma	&	\alpha\\
\alpha^{*}	&	1-\bar{\gamma}
\end{pmatrix}\leq\begin{pmatrix}
1	&	0\\
0	&	1
\end{pmatrix}, \,\ \Tr \gamma <\infty \Big\}.
\end{align*}
Note that each pair $(\gamma,\alpha)\in \cK$ represents the (generalized) one-body density matrix of a {\em quasi-free state} $\Gamma$ in the 
fermionic Fock space, with the particle number expectation $\Tr \gamma<\infty$. These quasi-free states can be obtained by 
rotating the original vacuum via a Bogoliubov transformation, and they can be interpreted as a generalized class of the 
standard Slater determinants used in the HF theory. We refer to \cite{BacLieSol-94} for a comprehensive introduction to the HFB 
theory.

In the following, we  denote by $\gamma(x,y)$ and $\alpha(x,y)$ the integral kernels of $\gamma$ and $\alpha$, respectively. In 
particular, $\rho_{\gamma}(x):=\gamma(x,x)$ is called the one-body particle density, which satisfies
$$
\int_{\R^3} \rho_{\gamma}(x) d x = \Tr \gamma.
$$
We also denote
\begin{equation*}\label{1.4a}
D(\rho_{\gamma},\rho_{\gamma}):=\int_{\R^{3}}\int_{\R^{3}}\frac{\rho_{\gamma}(x)\rho_{\gamma}(y)}{|x-y|}dxdy\ \ \text{and}\ 
\  \Ex(\gamma):=\int_{\R^{3}}\int_{\R^{3}}\frac{|\gamma(x,y)|^{2}}{|x-y|}dxdy.
\end{equation*}
Note that for $(\gamma,\alpha)\in\cK$, we have $\big\langle\sqrt{\rho_{\gamma}},\sqrt{-\Delta}\sqrt{\rho_{\gamma}}\big\rangle
\leq \Tr(\smash{\sqrt{-\Delta}\gamma})$ (see  \cite[Lemma 2.1]{LenLew-10}) and
\begin{align}
\label{inter-bdd1}&\Ex(\gamma) \le D(\rho_{\gamma},\rho_{\gamma})\lesssim\|\sqrt{\rho_{\gamma}}\|_{L^{12/5}}^{4},\\
\label{inter-bdd3}&\Ex(\alpha) \leq \frac{\pi}{2}\Tr(\smash{\sqrt{-\Delta}\alpha\alpha^{*}})\leq 
\frac{\pi}{2}\Tr(\smash{\sqrt{-\Delta}\gamma}),
\end{align}
where the upper bound of \eqref{inter-bdd1} follows from the Hardy--Littlewood--Sobolev inequality, and the lower bound of 
\eqref{inter-bdd3} is a consequence of the following Hardy--Kato inequality
\begin{align}\label{eq:HK}
|x|^{-1}\le \frac{\pi}{2}\sqrt{-\Delta}\  \text{ on }\, L^2(\R^3).
\end{align}
For given $m\ge 0$ and $\kappa>0$, the HFB energy functional of neutron stars and white dwarfs is given by
\begin{equation}\label{HFB-energy}
\begin{split}\cE_{m,\kappa}^{\rm HFB}(\gamma,\alpha)=&\Tr\big( (\sqrt{-\Delta+m^{2}}-m ) \gamma\big)-\frac{\kappa}{2} 
D(\rho_{\gamma},\rho_{\gamma})\\
& +\frac{\kappa} 2 \Ex(\gamma) -\frac \kappa 2 \Ex(\alpha).
\end{split}\end{equation}
One can note from \eqref{inter-bdd1} and \eqref{inter-bdd3} that each term of the energy functional $\cE_{m,\kappa}^{\rm 
HFB}(\gamma,\alpha)$ is finite if $(\gamma,\alpha)\in\cK$. The corresponding minimization problem for a given total mass 
$\lambda>0$ is then given by
\begin{align}\label{eq:Imk-def}
I_{m,\kappa}^{\rm HFB}(\lambda) &= \inf\big\{\cE_{m,\kappa}^{\rm HFB}(\gamma,\alpha):\, (\gamma,\alpha)\in \cK,\ \Tr 
\gamma = \lambda \big\}.
\end{align}

As discussed in \cite[Proposition 2.1]{LenLew-10}, we have the existence of the critical mass ($i.e,$ the Chandrasekhar limit):  
for $\kappa \in (0,4/\pi)$, there exists a unique number $\lambda^{\rm HFB}_\kappa \in (0,\infty)$, which is independent of 
$m\ge 0$, such that
\begin{equation}\label{eq:Chandrasekhar-limit}
I_{m,\kappa}^{\rm HFB}(\lambda) 
\begin{cases} >-\infty, \ &\text{if}\ \ 0<\lambda\leq\lambda^{\rm HFB}_\kappa;\\
= -\infty , \ & \text{if}\ \ \lambda>\lambda^{\rm HFB}_\kappa.
\end{cases}
\end{equation}
Here the upper bound $\kappa<4/\pi$ essentially follows from the Hardy--Kato inequality \eqref{eq:HK}, and this requirement is 
natural since  $\lambda^{\rm HFB}_\kappa=0$ holds for $\kappa>4/\pi$, see \cite[Remark 3]{LenLew-10}. Moreover, $\lambda^{\rm 
HFB}_\kappa$ is nonincreasing and continuous with respect to $\kappa$, and
\begin{equation}\label{eq:lambdaHFB}\lambda^{\rm HFB}_\kappa\sim (\tau_c/\kappa)^{3/2} \ \ \text{ as } \ \kappa\to 0,
\end{equation}
where the universal constant $\tau_c>0$ satisfies
\begin{equation}\label{eq:def-tauc}
\tau_c= \inf \Big\{\frac{2 \int_{\R^3} f(x)^{4/3} dx \Big( \int_{\R^3} f(x) dx \Big)^{2/3} }{D(f,f)}:\,  0\le f \in 
L^{4/3}(\R^3) \cap L^1(\R^3)  \Big\}.
\end{equation}
It is known from \cite{LieOxf-81} that the variational problem  \eqref{eq:def-tauc} has a minimizer, which is unique up to 
scaling, dilation and translation, and $\tau_c>0$ numerically satisfies $\tau_c\approx  2.677$, see \cite{LieYau-87}.

When $m>0$ and $0<\lambda<\lambda^{\rm HFB}_\kappa$, the existence of ground states for $I_{m,\kappa}^{\rm HFB}(\lambda)$ was 
established by Lenzmann and Lewin \cite{LenLew-10}. More precisely, it was proved in \cite[Theorem 1]{LenLew-10} that the 
following strict binding inequality holds
\begin{align*}
I_{m,\kappa}^{\rm HFB}(\lambda) < I_{m,\kappa}^{\rm HFB}(\lambda') + I_{m,\kappa}^{\rm HFB}(\lambda - \lambda'),\ \ \  \forall 
\ 0<\lambda'<\lambda<\lambda^{\rm HFB}_\kappa,
\end{align*}
which also implies  the relative compactness in $\cX$ of any minimizing sequence $\{(\gamma_{n},\alpha_{n})\}_{n\in\mathbb{N}}$ 
for $I_{m,\kappa}^{\rm HFB}(\lambda)$ by applying the concentration-compactness method. On the other hand, as explained by the 
authors in \cite[Remark 8]{LenLew-10}, they could not prove the existence of minimizers for $\lambda=\lambda^{\rm HFB}_\kappa$, 
due to a lack of coercivity, although the energy remains bounded from below. They further predicted that if $m>0$, then 
minimizers do not exist at the critical mass. Our first main result provides a justification for this conjecture.

\begin{thm}[HFB mimimizers at the critical mass]\label{thm:HFB-non-exist}
For $\kappa\in (0,4/\pi)$, suppose $I_{m,\kappa}^{\rm HFB}(\lambda) $ is defined in \eqref{eq:Imk-def}, and let 
$\lambda^{\rm HFB}_\kappa$ be the critical mass defined in \eqref{eq:Chandrasekhar-limit}. Then we 
have
\begin{align}\label{eq:Imkappa} I_{m,\kappa}^{\rm HFB}(\lambda^{\rm HFB}_\kappa)=-m\lambda^{\rm HFB}_\kappa,\quad \forall\, 
m\ge 0.
\end{align}
Moreover, $I_{0,\kappa}^{\rm HFB}(\lambda^{\rm HFB}_\kappa)$ has a minimizer, and every minimizing sequence is relatively 
compact in $\cX$. On the other hand, $I_{m,\kappa}^{\rm HFB}(\lambda^{\rm HFB}_\kappa)$ has no minimizers for every $m>0$.
\end{thm}

The proof of Theorem \ref{thm:HFB-non-exist} is based on the following key observation.

\begin{lem}[Existence-nonexistence alternative] \label{lem:HFB-alternative} Let $\kappa\in (0,4/\pi)$ and $0< \lambda \le 
\lambda^{\rm HFB}_\kappa$. If $I^{\rm HFB}_{0,\kappa}(\lambda)$ has a minimizer, then $I^{\rm HFB}_{m,\kappa}(\lambda)$ has no 
mimimizer for every $m>0$.
\end{lem}

The proof of Lemma \ref{lem:HFB-alternative} relies only on the following scaling property: if $(\gamma,\alpha)\in \cK$, then 
\begin{align}\label{scale-E0}
\cE_{0,\kappa}^{\rm HFB}(\gamma_{\beta},\alpha_{\beta}) &= \beta \cE_{0,\kappa}^{\rm HFB}(\gamma,\alpha),\,\ 
(\gamma_\beta,\alpha_\beta)\in \cK,\,\  \Tr \gamma_\beta=\Tr \gamma,
\end{align}
where $\gamma_{\beta}(x,y): = \beta^3 \gamma(\beta x, \beta y)$ and $\alpha_{\beta}(x,y) := \beta^3 \alpha(\beta x, \beta y)$ for 
$\beta > 0$. As already remarked in  \cite[Appendix B]{LenLew-10}, we deduce from \eqref{scale-E0} that $I_{0,\kappa}^{\rm HFB}(\lambda) =0$ for $0< \lambda \le 
\lambda^{\rm HFB}_\kappa$. We shall show that this identity can be extended to all $m>0$, $i.e.,$ $I_{m,\kappa}^{\rm HFB}(\lambda) = -m\lambda$ holds 
for all $m>0$, provided that $I^{\rm HFB}_{0,\kappa}(\lambda)$ has a minimizer. On the other hand, it is not difficult to see that if 
$I_{m,\kappa}^{\rm HFB}(\lambda)$ has a minimizer with $m>0$, then   $I_{m,\kappa}^{\rm HFB}(\lambda)>-m\lambda$.

The existence of minimizers for $I^{\rm HFB}_{0,\kappa}(\lambda^{\rm HFB}_\kappa)$ in Theorem \ref{thm:HFB-non-exist} is proved by 
applying the concentration-compactness method \cite{Lions1}, where the dichotomy scenario is ruled out by  Lemma 
\ref{lem:HFB-alternative}, together with the existence result from \cite{LenLew-10} for $I^{\rm HFB}_{m,\kappa}(\lambda)$ satisfying $m>0$ and 
$0<\lambda<\lambda^{\rm HFB}_\kappa$. Following the earlier results in \cite{LenLew-10}, we thus obtain the following  complete 
alternative picture:
\begin{itemize}
\item For sub-critical mass $0<\lambda<\lambda^{\rm HFB}_\kappa$, the variational problem $I^{\rm HFB}_{m,\kappa}(\lambda)$ has a minimizer for $m>0$, whereas $I^{\rm HFB}_{0,\kappa}(\lambda)$ does not have any minimizer.
	
\item For the critical mass $\lambda = \lambda^{\rm HFB}_\kappa$, the situation is reversed: $I^{\rm HFB}_{0,\kappa}(\lambda^{\rm 
HFB}_\kappa)$ has a minimizer, but $I^{\rm HFB}_{m,\kappa}(\lambda^{\rm HFB}_\kappa)$ does not  have any minimizer for $m>0$.
\end{itemize}

The existence result of $I^{\rm HFB}_{0,\kappa}(\lambda^{\rm HFB}_\kappa)$ also implies that the mapping $\kappa \to \lambda^{\rm 
HFB}_\kappa$ is {\em strictly decreasing}, while it is only known from  \cite{LenLew-10} that  $\lambda^{\rm HFB}_\kappa$ is 
nonincreasing and continuous in $\kappa\in(0, 4/\pi)$. This allows us to define the {\em critical coupling constant} $\lambda\mapsto \kappa^{\rm 
HFB}_\lambda$ as a dual version of $\kappa\mapsto \lambda^{\rm HFB}_\kappa$, $i.e.,$
\begin{align}\label{eq:kappa-critical-lambda}
\kappa^{\rm HFB}_\lambda &= \max \left\{ \kappa \in (0,4/\pi):\,  I^{\rm HFB}_{m,\kappa}(\lambda) >-\infty \text { 
for all }m\ge 0 \right\} \nn\\
&=  \max \left\{ \kappa \in (0,4/\pi):\, I^{\rm HFB}_{0,\kappa}(\lambda)=0 \right\} .
\end{align}
We can also  characterize $\kappa_\lambda^{\rm HFB}$ by the following Gagliardo--Nirenberg type inequality:

\linespread{1.0}
\begin{thm}[Gagliardo--Nirenberg inequality of HFB problem] \label{lem:GN-HFB} Denote the limit
$\lim_{\kappa \nearrow 4/\pi} \lambda^{\rm HFB}_\kappa:=\lambda^*.$ Then the mapping $\lambda\to \kappa^{\rm HFB}_\lambda$ 
is a strictly decreasing, continuous, and bijective function from $(\lambda^*,\infty)$ to $(0,4/\pi)$, and $\kappa^{\rm 
HFB}_\lambda\sim \tau_c \lambda^{-2/3}$ as $\lambda\to \infty$, where $\tau_c$ is given in \eqref{eq:def-tauc}. Moreover, we 
have for every $\lambda\in (\lambda^*,\infty)$,
\begin{align}\label{eq:GN-HFB}
\kappa^{\rm HFB}_\lambda = \inf \Big\{ \frac{2\Tr (\sqrt{-\Delta}\gamma)}{D(\rho_\gamma,\rho_\gamma)-X(\gamma)+X(\alpha)}:\,  (\gamma,\alpha)\in \cK,\,\Tr \gamma = \lambda\Big\},
\end{align}
which admits a minimizer.
\end{thm}

Theorem \ref{lem:GN-HFB} gives  the asymptotic formula $\kappa^{\rm HFB}_\lambda\sim \tau_c \lambda^{-2/3}$ as $\lambda\to 
\infty$, which is consistent with the many-body  analogue
\begin{align}\label{eq:GN-Q}
\inf \Big\{ \frac{ \langle \Psi, \sum_{i=1}^N \sqrt{-\Delta_{x_i}} \Psi\rangle}{\langle \Psi, \sum_{i<j} |x_i-x_j|^{-1} 
\Psi\rangle}:\, \Psi \in \bigwedge^NL^{2}(\R^3),\  \|\Psi\|_{L^2}^2=1\Big\} \sim \tau_cN^{-3/2} 
\end{align}
as $N\to\infty$ due to Lieb and Yau \cite{LieYau-87}. Recently, the formula \eqref{eq:GN-Q} was extended by Frank, Hoffmann-Ostenhof, Laptev 
and Solovej \cite{manybodyhd2} to cover the fractional Laplacian $(-\Delta)^s$ with $0<s\le 1$ and  the corresponding Riesz 
potential $|x-y|^{-2s}$; see also \cite{manybodyhd} for earlier results on many-body Hardy inequalities. Unlike the nonlinear 
variational problem \eqref{eq:GN-HFB}, the infimum of \eqref{eq:GN-Q} has no minimizers, since the energy functional is 
translation-invariant.

\subsection{Hartree--Fock theory} 

We now focus on the HF theory. We first consider an analogue of \eqref{eq:GN-HFB} restricted to the projections $\gamma = 
\gamma^2$ and $\alpha = 0$, or equivalently we consider \eqref{eq:GN-Q} only for Slater determinants. Denote
\begin{equation}\label{set}
\cP_{N}:=\Big\{\gamma=\sum\limits_{j=1}^{N}|u_{j}\rangle\langle u_{j}|:\,	u_{j}\in H^{1/2}(\R^3),\ \langle 
u_{j},u_k\rangle=\delta_{jk},\ j, k=1,\cdots, N\Big\}.
\end{equation}
We establish the following Gagliardo--Nirenberg inequality in the HF theory.

\begin{thm}[Gagliardo--Nirenberg inequality for HF problem]\label{th2.1} For every $2\le N\in\mathbb{N}$, the variational problem
\begin{align}\label{eq:GN-HF}
\kappa^{\rm HF}_N:= \inf_{\gamma \in \cP_N} \frac{2\Tr (\sqrt{-\Delta} \gamma)}{D(\rho_\gamma,\rho_\gamma)-X(\gamma)}
\end{align}
admits a minimizer. Moreover, we have
\begin{enumerate}
\item[(1)]  The sequence $\{\kappa^{\rm HF}_N\}_{N\ge 2}$ is strictly decreasing in $N$, and  $\kappa^{\rm HF}_N\sim 
\tau_c N^{-2/3}$ as $N\to \infty$, where $\tau_c>0$ is given in \eqref{eq:def-tauc}.
		
\item[(2)]  Any minimizer $\gamma$ of (\ref{eq:GN-HF}) can be written in the form 
$\gamma=\sum_{j=1}^{N}|w_{j}\rangle\langle w_{j}|$, where the orthonormal  functions $w_1, \cdots, w_N$ are the 
eigenfunctions of
\begin{equation}\label{h8} H_{\gamma}:=\sqrt{-\Delta}-\kappa^{\rm HF}_N\big(\rho_{\gamma}\ast|\cdot|^{-1}\big)(x) 
+\kappa^{\rm HF}_N\frac{\gamma(x,y)}{|x-y|}\ \,\ \text{on} \ \, L^2(\R^3),
\end{equation}
together with negative eigenvalues $\nu_{1}\leq\nu_{2}\leq\cdots \leq \nu_{N}<0$, $w_{j}\in C^\infty(\R^3)$ satisfies the decaying property
\begin{equation}\label{decay}
\rho_{\gamma}(x)=\sum_{j=1}^{N}|w_{j}(x)|^{2}\leq C\big(1+|x|\big)^{-8}\ \ \text{in}\ \  \R^3.
\end{equation}
\end{enumerate}
\end{thm}

\linespread{1.5}

The proof of Theorem \ref{th2.1} is based on the insights from the HFB theory, but we face the extra challenges in implementing 
an induction argument, because of the nonlinear constraint $\gamma = \gamma^2$.  As we shall illustrate, the constant 
$\kappa^{\rm HF}_N$ of \eqref{eq:GN-HF} coincides with the optimal constant of the following interpolation inequality
\begin{equation}\label{orthonormal}
\|\gamma\|\mathrm{Tr}\big(\sqrt{-\Delta}\, \gamma\big)\geq\frac{\kappa^{\rm HF}_N}{2}
\int_{\R^{3}}\int_{\R^{3}}\frac{\rho_{\gamma}(x)\rho_{\gamma}(y)-|\gamma(x, y)|^{2}}{|x-y|}dxdy,
\ \  \forall\, \gamma\in \mathcal{R}_{N},
\end{equation}
where
\begin{equation}\label{rn}
\mathcal{R}_{N}:=\Big\{ \gamma:\, \gamma=\gamma^* \ge 0, \ \mathrm{Rank}(\gamma)\leq N, \ \mathrm{Tr}\big(\sqrt{-\Delta}\, 
\gamma\big)<\infty\Big\}.
\end{equation}
The inequality \eqref{orthonormal} can be interpreted as an extension of the following Gagliardo--Nirenberg inequality for 
relativistic boson stars:
\begin{equation*}\label{gnt}
\|u\|_{2}^{2}\big\langle u, \sqrt{-\Delta}u\big\rangle \geq \frac{\kappa}{2}\int_{\R^{3}}\int_{\R^{3}} 
\frac{|u(x)|^{2}|u(y)|^{2}}{|x-y|}dxdy, \,\ \forall\, u\in H^{1/2}(\R^3),
\end{equation*}
which was recently discussed by Lieb and Yau \cite{LieYau-87}, see also \cite{gnt,2009unique,2007fr0hlich,gnt,nguyen,Guo17} and 
the references therein.

To prove Theorem \ref{th2.1}, we first study the existence of optimizers for the relaxed problem (\ref{orthonormal}), which is 
easier than \eqref{eq:GN-HF} with the nonlinear constraint $\gamma=\gamma^2$. This idea goes back to the work of Lieb and Simon  
\cite{LieSim-77} on the standard Hartree--Fock theory for Coulomb systems, and the choice $\mathrm{Rank}(\gamma)\leq N$ in the 
variational set \eqref{rn} is inspired by the recent work of Frank, Gontier and Lewin
\cite{ii,FGL} on the best constant $K_{N}>0$ of the following finite-rank Lieb-Thirring inequality
\begin{equation*}\label{eq:FGL}
\|\gamma\|^{2/3}\Tr(-\Delta \gamma)
\geq K_{N} \int_{\R^{3}}\rho_\gamma(x)^{\frac{5}{3}} dx,\ \ \forall\, \gamma\in\mathcal{R}_{N}.
\end{equation*}
In order to deduce that the optimizer $\gamma $ of \eqref{orthonormal} is an $N$-dimensional orthogonal projection, we need to 
prove the strict monotonicity
\begin{equation}\label{1.12a}
+\infty:=\kappa_{1}^{\rm HF} >	\kappa_{2} ^{\rm HF} >\kappa_{3}^{\rm HF}>\kappa_{4}^{\rm HF}>\cdots>\kappa^{\rm 
HF}_N>\cdots>0.
\end{equation}
This key step of proving (\ref{1.12a}) is established in Lemma \ref{nomin}  by an induction argument, using an analogue of 
Lemma \ref{lem:HFB-alternative} on the existence-nonexistence alternative of  Hartree--Fock minimizers with $m>0$ and $m=0$, 
together with the existence result in \cite{LenLew-10}.
It is interesting that the strict monotonicity of \eqref{1.12a} is completely different from \cite[Theorem 6]{ii}, where the 
function $N\mapsto K_{N}$ may not be strictly decreasing, and hence the picture in the present work is much clearer.

The decaying estimate \eqref{decay} is of independent interest, which is useful to discuss the refined profiles of minimizers 
for (\ref{eq:GN-HF}). We remark that \eqref{decay} is somewhat similar to that of the following relativistic Hartree problem 
for bosons:
$$\sqrt{-\Delta}w-\int_{\R^{3}}\frac{w(y)}{|x-y|}dy\, w=-w \ \ \mbox{in} \,\ \R^3,$$
which was addressed by Frank and Lenzmann in \cite{gnt}. However, our system is more involved, because the mean-field operator 
$H_{\gamma }$ defined in (\ref{h8}) has an exchange term. The decaying estimate \eqref{decay} is proved by using Hardy-Kato 
inequality (cf. \cite{Kato-72}), in the spirit of the ``Schr\"odinger inequality" approach from M.  Hoffmann-Ostenhof and T. 
Hoffmann-Ostenhof \cite{ho}. Moreover, we handle directly the exchange term with a pointwise estimate, which differs from the 
standard decaying analysis in  \cite{LieSim-77}. We expect that this alternative approach is also useful in other contexts.

For $m\ge 0$ and $\kappa>0$, we now consider the HF variational problem
\begin{equation}\label{problem}
\begin{split}
I^{\rm HF}_{m,\kappa}(N):=\inf\limits_{\gamma\in\cP_{N}}\cE^{\rm HF}_{m,\kappa}(\gamma),
\ \ \kappa>0,\ 2\leq N\in\mathbb{N},
\end{split}
\end{equation}
where
\begin{equation*}\label{functional}
\cE^{\rm HF}_{m,\kappa}(\gamma):=\mathrm{Tr}\Big( \big(\sqrt{-\Delta+m^{2}}-m\big)\gamma\Big)  - 
\frac{\kappa}{2}\int_{\R^{3}}\int_{\R^{3}}\frac{\rho_{\gamma}(x)\rho_{\gamma}(y)-|\gamma(x,y)|^{2}}{|x-y|}dxdy.
\end{equation*}
As in \cite[Section 3.3]{LenLew-10}, we denote by $N^{\rm HF}_\kappa$ the largest integer such that $I^{\rm HF}_{m,\kappa}(N) > 
-\infty$. This number is finite and independent of $m \geq 0$. We then get from \cite[Theorem 4]{LenLew-10} that for every $m > 
0$ and $\kappa > 0$, the minimizers of $I^{\rm HF}_{m,\kappa}(N)$ exist for any $2 \leq N < N^{\rm HF}_\kappa$. As a 
consequence of Theorem \ref{th2.1}, we obtain the following complete characterization on the existence and non-existence of 
minimizers at the critical mass $N^{\rm HF}_\kappa$.

\linespread{0.7}
\begin{thm}[HF mimimizers at the critical mass]\label{cor1}
For $m>0$, let $\kappa^{\rm HF}_N$ be defined in \eqref{eq:GN-HF}. Then for every $\kappa>0$, the HF variational problem 
$I^{\rm HF}_{m,\kappa}(N^{\rm HF}_\kappa)$ in \eqref{problem} has a minimizer, if and only if
$$\kappa \notin \{\kappa_N^{\rm HF}\}_{N\ge 2} \cup (\kappa_2^{\rm HF},\infty).$$
More precisely, we have the following conclusions:
\begin{enumerate}
\item[(1)] If $\kappa = \kappa^{\rm HF}_N$ holds for some $2 \leq N \in \mathbb{N}$, or $\kappa > \kappa_2^{\rm HF}$, then 
$N^{\rm HF}_\kappa = N$ or $N^{\rm HF}_\kappa = 1$, respectively, and $I^{\rm HF}_{m,\kappa}\big(N^{\rm HF}_\kappa\big)$ 
has no minimizer  in both cases.
		
\item[(2)]  If $\kappa\in (\kappa_{N+1}^{\rm HF}, \kappa^{\rm HF}_N)$ holds for some $2\leq N\in\mathbb{N}$, then $N^{\rm 
HF}_\kappa=N$, and  $I^{\rm HF}_{m,\kappa}\big(N^{\rm HF}_\kappa\big)$ admits at least one minimizer.
\end{enumerate}
\end{thm}

As a dual version of Theorem \ref{cor1}, we may also address the existence and nonexistence of minimizers for $I^{\rm 
HF}_{m,\kappa}(N)$ defined in (\ref{problem}), where $m > 0$ and $2 \le N \in \mathbb{N}$ are fixed, while $\kappa > 0$ is 
treated as a varying parameter. More precisely, we have
\begin{itemize}
\item If $0<\kappa<\kappa^{\rm HF}_N$, then it is shown in \cite[Theorem 4]{LenLew-10} that $I^{\rm HF}_{m,\kappa}(N)$ has 
a minimizer, and any minimizer $\gamma$ of $I^{\rm HF}_{m,\kappa}(N)$ can be written in the form  $\gamma=\sum_{j=1} 
^N|u_{j}\rangle \langle u_{j}|$, where $\langle u_{i}, u_{j}\rangle =\delta_{ij}$ and
\begin{align}\label{equation}
H_{\gamma}u_{j}=\nu_{j} u_{j}\ \ \text{in}\, \ \R^3, \ \ j=1,\cdots, N.
\end{align}
Here $\nu_{1} \leq \nu_{2}\leq\cdots\leq\nu_{N}<0$ are $N$ negative eigenvalues of the mean field operator
\begin{equation}\label{operator}
H_{\gamma}:=\sqrt{-\Delta+m^{2}}-m-\kappa(|x|^{-1}\ast\rho_{\gamma})+\kappa \frac{\gamma(x,y)}{|x-y|}\ \ \text{on} \ \, L^2(\R^3).
\end{equation}

\item If $\kappa=\kappa^{\rm HF}_N$, then $I^{\rm HF}_{m,\kappa}(N)=-mN$ and $I^{\rm HF}_{m,\kappa^{\rm HF}_N}(N)$ has no minimizer, see Lemma \ref{nomin} below.

\item If $\kappa >  \kappa^{\rm HF}_N$, then $I^{\rm HF}_{m,\kappa}(N)=-\infty$.

\end{itemize}

Motivated by above results, we now consider the minimizers of $I^{\rm HF}_{m,\kappa}(N)$ as $\kappa \nearrow \kappa^{\rm 
HF}_N$, where the constant $\kappa^{\rm HF}_N>0$ is given by \eqref{eq:GN-HF}. We shall show that the limiting behavior can be 
fully described by the optimizer of the Gagliardo-Nirenberg inequality \eqref{eq:GN-HF} given in Theorem \ref{th2.1}.

\begin{thm}[Limiting profile of HF minimizers]\label{th2}
For $m>0$ and $2\leq N\in\mathbb{N}$, let $\gamma_{\kappa_{n}}=\sum_{j=1}^{N}|u^{\kappa_{n}}_{j}\rangle \langle 
u^{\kappa_{n}}_{j}|$ be a minimizer of $I_{m,\kappa_{n}}^{\rm HF}(N)$ satisfying $\kappa_{n}\nearrow \kappa^{\rm HF}_N$ as 
$n\to\infty$, where $(u_1^{\kappa_{n}}, \cdots, u_{N}^{\kappa_{n}})$ satisfies \eqref{equation}. Then there exist a subsequence 
of $\{\gamma_{\kappa_{n}}\}$, still denoted by $\{\gamma_{\kappa_{n}}\}$,  and a sequence $\{z_n\}\subset\R^3$ such that for 
all $j=1, \cdots, N,$
\begin{equation}\label{1.23a}
\epsilon_{\kappa_{n}}^{3/2}u_{j}^{\kappa_{n}}(\epsilon_{\kappa_{n}}x+\epsilon_{\kappa_{n}}z_n) \rightarrow w_{j}(x) \ \ \text{strongly\ in}\ \ H^{1/2}(\R^{3})\cap L^{\infty}(\R^{3})
\end{equation}
as $n\rightarrow\infty$, where  $\gamma^{*}:=\sum_{j=1} 
^N|w_{j}\rangle \langle w_{j}|$ is an optimizer of the variational problem
\begin{equation}\label{dn}
d_{N}^{*}:=\inf\Big\{\Tr\left( (-\Delta)^{-1/2} \gamma \right):\, \gamma\ \text{is\ an\ optimizer\ of}\ 
\eqref{eq:GN-HF},\ \Tr\big(\sqrt{-\Delta}\, \gamma\big)=1 \Big\},
\end{equation}
and
\begin{equation}\label{1.12b}
\begin{split}
\epsilon_{\kappa_{n}}:=\frac{1}{\Tr(\sqrt{-\Delta}\, \gamma_{\kappa_{n}})}=\Big[\frac{2( \kappa^{\rm 
HF}_N-\kappa_{n})}{m^2 \kappa^{\rm HF}_N d_N^* }\Big]^{1/2}\big[1+o(1)\big]\ \ as \ \, n\to\infty.
\end{split}\end{equation}
\end{thm}	

Theorem \ref{th2} presents the mass concentration of minimizers for $I_{m,\kappa_{n}}^{\rm HF}(N)$ as $\kappa_{n}\nearrow 
\kappa^{\rm HF}_N$.
Although the limiting behavior of \eqref{1.23a} is similar to that of bosonic systems already addressed in \cite{Guo17}, where 
the Newton potential $|x|^{-1}$ is replaced by a function $K(x)/|x|$   satisfying
$\lim_{|x|\to0}(1-K(x))|x|^{-p}\in(0,\infty)$ for some $p>1$, both the exchange term and the Newton potential 
$|x|^{-1}$
of $I_{m,\kappa_{n}}^{\rm HF}(N)$ lead to extra difficulties in analyzing the blow-up rate $\epsilon_{\kappa_{n}}^{-1}$.
A key ingredient for the proof of Theorem \ref{th2} is a detailed investigation of the limiting optimization problem $d_N^*$, 
which is given in Lemma \ref{lem3.1a}. In particular, we  use the decaying property  \eqref{decay}, together with the 
Hardy--Kato inequality \eqref{eq:HK}, to deduce that
\begin{align}\label{eq:HK-decay}
\Tr \big( (-\Delta)^{-1/2}\gamma\big) \le C \Tr (|x|\gamma) = C\int_{\R^3} |x|\rho_\gamma(x) dx <\infty
\end{align}
holds for any optimizer $\gamma$ of $\kappa_N^{\rm HF}$. We are then able to establish in Lemma \ref{lem3.1} the following 
energy estimate
\begin{align}\label{1.12}
&I^{\rm HF}_{m,\kappa_{n}}(N)+mN \sim  m \Big[\frac{2d_N^*(\kappa^{\rm HF}_N-\kappa_{n})}{\kappa^{\rm HF}_N}\Big]^{1/2} \to 
0\ \ \ \mbox{as}\ \ n\rightarrow\infty,
\end{align}
which eventually implies the blow-up rate $\epsilon_{\kappa_{n}}^{-1}$ in \eqref{1.12b}.

\subsection*{Organization of the proof} The rest of the paper is devoted to the proofs of all above results. In Section 
\ref{sec:HFB}, we focus on the Hartree-Fock-Bogoliubov theory, where we address the proofs of Lemma \ref{lem:HFB-alternative}, 
Theorem \ref{thm:HFB-non-exist}, and Theorem \ref{lem:GN-HFB}. We then turn to the Hartree-Fock theory, for which we shall 
complete in Section \ref{sec:GN-ineq-ON} the proofs of Theorems \ref{th2.1} and \ref{cor1}. Finally,   Theorem \ref{th2} is 
proved in Section \ref{sec:pf-lim-min}.

\subsection*{Acknowledgments} P. T. Nam  would like to thank Rupert L. Frank for helpful remarks on the ``Schr\"odinger inequality" approach from \cite{ho} and Quoc Hung Nguyen for his warm hospitality during a visit in 2024 
to the Academy of Mathematics and Systems Science (CAS), where part of this work was carried out. Y. J. Guo is partially 
supported by the National Key R \& D Program of China (Grant 2023YFA1010001) and the NSF of China (Grants 12225106 and 
12371113). P. T. Nam and D. H. Ou Yang are partially supported by the European Research Council (ERC Consolidator Grant RAMBAS 
Project Nr. 10104424) and the Deutsche Forschungsgemeinschaft (TRR 352 Project Nr. 470903074).

\section{Hartree--Fock--Bogoliubov minimizers}\label{sec:HFB}

In this section, we prove all results on the Hartree-Fock-Bogoliubov  (HFB) theory. We start with the proof of Lemma \ref{lem:HFB-alternative} in Section \ref{sec:31}, and the proofs of Theorems \ref{thm:HFB-non-exist} and \ref{lem:GN-HFB} are then given in Section \ref{sec:32}. For this purpose, in this section we always fix $\kappa\in (0,4/\pi)$ and denote $\lambda^{\rm HFB}_\kappa$ as the critical mass defined in \eqref{eq:Chandrasekhar-limit}.

\subsection{Proof of Lemma \ref{lem:HFB-alternative}} \label{sec:31}
\begin{proof}
We first observe that if $I^{\rm HFB}_{0,\kappa}(\lambda)$ has a minimizer, then we have for all $m>0$,
\begin{equation}\label{3:3.1}
	I^{\rm HFB}_{m,\kappa}(\lambda)=-m\lambda .
\end{equation}
In fact, since $\sqrt{-\Delta+m^2}\ge \sqrt{-\Delta}$ and $I_{0,\kappa}^{\rm HFB}(\lambda)=0$ (cf. \cite[Appendix B]{LenLew-10}), we have the obvious lower bound $I_{m,\kappa}^{\rm HFB}(\lambda) \ge -m\lambda$.

On the other hand, consider a minimizer $(\gamma,\alpha)$ of $I_{0,\kappa}^{\rm HFB}(\lambda)$. Define for $n\in \mathbb{N}$,
$$	\gamma_{n}(x,y):=n^3 \gamma(nx,ny), \ \ \alpha_{n}(x,y):=n^{3}\alpha(nx,ny).$$
We then derive from \eqref{scale-E0} that $(\gamma_n,\alpha_n)\in \cK$, $\Tr (\gamma_n)=\Tr \gamma=\lambda$ and $\cE_{0,\kappa}^{\rm HFB}(\gamma_n,\alpha_n)=n \cE_{0,\kappa}^{\rm HFB}(\gamma,\alpha)=0$. Therefore,
\begin{align}\label{E(gn,an)}
I^{\rm HFB}_{m,\kappa}(\lambda) \le \cE_{m,\kappa}^{\rm HFB}(\gamma_n,\alpha_n)&=\cE_{0,\kappa}^{\rm HFB}(\gamma_n,\alpha_n)-m \Tr (\gamma_{n}) +\Tr(B_{m}\gamma_n) \nonumber\\
&= -m\lambda +\Tr(B_{m,n}\gamma),
\end{align}
where
\begin{align*}
B_{m} &= \sqrt{-\Delta+m^2} - \sqrt{-\Delta}= \frac{m^2}{ \sqrt{-\Delta+m^{2}}+\sqrt{-\Delta}}>0, \nonumber\\
B_{m,n}&= \frac{m^2}{ \sqrt{-n^2 \Delta+m^{2}}+\sqrt{-n^2 \Delta}}\le m.
\end{align*}
By the Fourier transform, we then have $\langle u, B_{m,n}u\rangle \to 0$ as  $n\to \infty$ for every $u\in L^2(\R^3)$.  Since $\gamma$ is in the trace class, it further implies that
$\Tr(B_{m,n} \gamma) \to 0$ as $n\to \infty$, due to the Dominated Convergence Theorem. Thus, we deduce from \eqref{E(gn,an)} that $I_{m,\kappa}^{\rm HFB}(\lambda)\le -m\lambda$, and hence (\ref{3:3.1}) holds ture.

Moreover, since $\cE_{0,\kappa}^{\rm HFB}(\widetilde \gamma,\widetilde \alpha) \ge I^{\rm HFB}_{0,\kappa}(\lambda)=0$ holds  for every $(\widetilde \gamma,\widetilde \alpha)\in \cK$ satisfying $\Tr (\widetilde \gamma)=\lambda$,  we have
$$
\cE_{m,\kappa}^{\rm HFB}(\widetilde \gamma,\widetilde \alpha)= \cE_{0,\kappa}^{\rm HFB}(\widetilde \gamma,\widetilde \alpha) - m \lambda + \Tr (B_m \widetilde \gamma) > -m \lambda\ \ \text{for every} \ (\widetilde \gamma,\widetilde \alpha)\in \cK,
$$
where $B_m$ as before  is a strictly positive operator on $L^2(\R^3)$  with a trivial kernel.
This implies that $\cE_{m,\kappa}^{\rm HFB}(\lambda)$ has no minimizer for any $m>0$, and the proof of Lemma \ref{lem:HFB-alternative} is therefore complete.
\end{proof}

\subsection{Proofs of Theorems \ref{thm:HFB-non-exist} and  \ref{lem:GN-HFB}}\label{sec:32}
The aim of this subsection is to prove Theorems \ref{thm:HFB-non-exist} and  \ref{lem:GN-HFB}.
Following the existence result (cf.    \cite{LenLew-10}) for the sub-critical mass case with $m > 0$, we first address the proof of Theorem \ref{thm:HFB-non-exist} by applying Lemma \ref{lem:HFB-alternative} and the concentration-compactness method \cite{Lions1}.

\begin{proof}[Proof of Theorem \ref{thm:HFB-non-exist}] We first prove that  the problem $I^{\rm HFB}_{0,\kappa}(\lambda^{\rm HFB}_\kappa)$ admits at least one minimizer, where $\lambda^{\rm HFB}_\kappa>0$ denotes the critical mass  defined in \eqref{eq:Chandrasekhar-limit}. By the scaling property \eqref{scale-E0}, we can choose a minimizing sequence $\{(\gamma_n,\alpha_n)\}$ of $I^{\rm HFB}_{0,\kappa}(\lambda^{\rm HFB}_\kappa)=0$ such that $\Tr(\sqrt{-\Delta}\gamma_n)=1$ holds for all $n\in \mathbb{N}$.
	
Following  \cite{Lions1}, we define
\begin{equation*}\label{con-fun}
Q_{n}(R) := \sup_{y \in \R^{3}} \int_{|x - y| \leq R} \rho_{\gamma_{n}}(x) dx, \,\ R \in (0, \infty).
\end{equation*}
For every $n \in \mathbb{N}$, the function $R \mapsto Q_n(R)$ is increasing, and
$$\lim_{R\to \infty}Q_n(R) = \Tr \gamma_n =  \lambda^{\rm HFB}_\kappa.$$
Therefore, up to a subsequence, we obtain from Helly's selection theorem that ${Q_{n}(R)}$ has a pointwise limit $Q(R) := \lim_{n \to \infty} Q_{n}(R)$. Set
\begin{align}\label{lim-con-fun}
\lambda^{1}:=\lim_{R\rightarrow\infty}Q(R)\in [0,\lambda^{\rm HFB}_\kappa].
\end{align}
As explained in \cite[Lemma 1.1]{Lions1}, we have the following three possibilities:
\begin{enumerate}
\item[(i)] (Compactness) $\lambda^{1}=\lambda^{\rm HFB}_\kappa$;
\item[(ii)] (Vanishing) $\lambda^{1}=0$;
\item[(iii)] (Dichotomy) $0<\lambda^{1}< \lambda^{\rm HFB}_\kappa$.
\end{enumerate}
We now rule out the vanishing case.

\begin{lem}\label{prop:no-vanishing}
The constant $\lambda^1 $ of (\ref{lim-con-fun}) satisfies $\lambda^1 > 0$.
\end{lem}

\begin{proof}
On the contrary, suppose that $\lambda^1=0$. Applying \cite[Lemma 7.2]{LenLew-10}, it then yields that $\|\sqrt{\rho_{\gamma_n}}\|_{L^p}\rightarrow 0$ as $n\to\infty$ for any $2<p<3$. We thus deduce from \eqref{inter-bdd1} that
\begin{align}\label{eq:non-vanishing-HFB}
0\leq \Ex(\gamma_{n}) \le D(\rho_{\gamma_{n}},\rho_{\gamma_{n}})\lesssim\|\sqrt{\rho_{\gamma_{n}}}\|_{L^{12/5}}^{4}\rightarrow 0\ \ as\ \ n\to\infty.
\end{align}
We further derive from \eqref{inter-bdd3} that
\begin{equation*}\label{vanish-lim}
\begin{split}\cE_{0,\kappa}^{\rm HFB}(\gamma_{n},\alpha_{n}) &= \Tr(\smash{\sqrt{-\Delta}\gamma_{n}})-\frac{\kappa}{2}\int_{\R^{3}}\int_{\R^{3}}\frac{|\alpha_{n}(x,y)|^{2}}{|x-y|}dxdy + o(1) \\
&\ge \big(1-\kappa\pi/4\big)\Tr(\smash{\sqrt{-\Delta}\gamma_{n}}) + o(1) \\
&=1-\kappa\pi/4+o(1) >0 \ \ \ \text{as}\ \ n\to\infty,
\end{split}
\end{equation*}
since $\kappa<4/\pi$ and $\Tr(\smash{\sqrt{-\Delta}\gamma_{n}})=1$. This however  contradicts with the assumption that $\cE_{0,\kappa}^{\rm HFB}(\gamma_{n},\alpha_{n}) \to 0$ as $n\to\infty$. Therefore, it holds that $\lambda^1 > 0$, and we are done.
\end{proof}

For $0<\lambda^{1}\le \lambda^{\rm HFB}_\kappa$, we now extract a nontrivial weak limit of $(\gamma_n,\alpha_n)$. Consider smooth cutoff functions $\chi,\eta:\R^3\to [0,1]$ satisfying
$$\chi^{2}+\eta^{2}=1,\,\ \one_{B_1}\le \chi \le \one_{B_2},$$
where $\one_{B_R}$ denotes the indicator function of the ball $B_R=\{|x|<R\}$. For each $R>0$, we define
\begin{align}\label{eq:def-chiR-etaR}\chi_{R}(x)=\chi(x/R),\,\ \eta_{R}(x)=\eta(x/R).
\end{align}
Applying \eqref{lim-con-fun}, we then have the following result.

\vspace{2mm}

\begin{lem}[Localized minimizing sequences]\label{lem:str-loc}
There exist $\{R_{n}\}\subset (0,\infty)$, which satisfies $R_{n}\rightarrow\infty$ as $n\to\infty$, and $\{y_{n}\}\subset \R^{3}$ such that
\begin{equation*}
\begin{split}
&(-\Delta)^{1/4}\tau_{y_{n}}\chi_{R_{n}}\gamma_{n}\chi_{R_{n}}\tau_{y_{n}}^{*}(-\Delta)^{1/4}\\
\rightharpoonup &\,(-\Delta)^{1/4}\gamma(-\Delta)^{1/4}\ \  \text{weakly-*\ in}\ \fS_{1}\ \ as\ \ n\to\infty,
\end{split}\end{equation*}
and
\begin{align}\label{3.9}
\tau_{y_{n}}^{*}\chi_{R_{n}}\gamma\chi_{R_{n}}\tau_{y_{n}}\rightarrow\gamma^1 \ \ \text{strongly\ in} \ \fS_{1}\ \ as\ \ n\to\infty,
\end{align}
where 
$\Tr(\smash{\gamma^1})=\lambda^{1}$,  $(\tau_y f)(x)=f(x-y)$ and $\fS_{1}$ denotes the space of trace-class operators.  Moreover, we have
\begin{equation}\label{eq:localization-estimate-key}
\begin{split}&\lim_{n\rightarrow\infty}\int_{R_{n}\leq |x-y_{n}|\leq 2R_{n}}\rho_{\sqrt{-\Delta}\gamma_{n}\sqrt{-\Delta}}(x)dx\\
&=\lim_{n\rightarrow\infty}\int_{R_{n}\leq |x-y_{n}|\leq 2R_{n}}\rho_{\gamma_{n}}(x)=0.
\end{split}\end{equation}
\end{lem}

Because the proof of Lemma \ref{lem:str-loc} is similar to that of \cite[Lemma 7.3]{LenLew-10}, we omit its details for simplicity. Since the energy functional $\cE_{0,\kappa}^{\rm HFB}$ is translational invariant, without loss of generality, we may assume that the sequence $\{y_{n}\}\subset \R^{3}$ of Lemma \ref{lem:str-loc} satisfies $y_{n}= 0$ for all $n\geq 0$. Also, since $\Tr (\alpha_n \alpha_n^*)\le \Tr (\gamma_n)=\lambda^{\rm HFB}_\kappa$ and $\Tr (\sqrt{-\Delta}\alpha_n \alpha_n^*)\le \Tr (\sqrt{-\Delta}\gamma_n)=1$,  we can assume that up to a subsequence, $\chi_{R_{n}} \alpha_n \rightharpoonup \alpha^1$ weakly in $\fS_{2}$ as $n\to\infty$, and
$$(-\Delta)^{1/4}\chi_{R_{n}}\alpha_{n}\alpha_n^*\chi_{R_{n}}(-\Delta)^{1/4}\rightharpoonup(-\Delta)^{1/4}\alpha^1 (\alpha^1)^*(-\Delta)^{1/4}$$
weakly-* in $\fS_{1}$ as $n\to\infty$, where the sequence $\{R_{n}\}\subset (0,\infty)$ is as in Lemma \ref{lem:str-loc}. Define
\begin{align}\label{eq:gamman1-gamman2} (\gamma_n^{(1)},\alpha_{n}^{(1)}):=\chi_{R_{n}}(\gamma_{n},\alpha_{n})\chi_{R_{n}},\ \ (\gamma_n^{(2)},\alpha_{n}^{(2)}):=\eta_{R_{n}}(\gamma_{n},\alpha_{n})\eta_{R_{n}}.
\end{align}
By \eqref{eq:localization-estimate-key} and the IMS localization formulas (cf. \cite[Lemma A.1]{LenLew-10}), we can split the mass
\begin{align}\label{eq:HFB-splitting-mass}
\lambda^{\rm HFB}_\kappa = \lambda^1 + (\lambda^{\rm HFB}_\kappa - \lambda^1) = \Tr \gamma_n^{(1)} + \Tr \gamma_n^{(2)} + o(1)\ \ \mbox{as}\ \, n\to \infty,
\end{align}
and split the total energy
\begin{equation}\label{eq:HFB-splitting-E}
\begin{split}\cE_{0,\kappa}^{\rm HFB} (\gamma_n,\alpha_n) 
\geq&\cE_{0,\kappa}^{\rm HFB} (\gamma_n^{(1)},\alpha_{n}^{(1)}) \\
&+  \cE_{0,\kappa}^{\rm HFB}  ( \gamma_n^{(2)},\alpha_{n}^{(2)} ) + o(1) \ \ \mbox{as}\ \, n\to \infty.
\end{split}\end{equation}
Since $ \cE_{0,\kappa}^{\rm HFB} (\gamma_n^{(1)},\alpha_{n}^{(1)})\geq0,\  \cE_{0,\kappa}^{\rm HFB}  ( \gamma_n^{(2)},\alpha_{n}^{(2)} ) \ge 0$ and $\cE_{0,\kappa}^{\rm HFB} (\gamma_n,\alpha_n) \to 0$ as $n\to\infty$, we obtain that
\begin{align}\label{eq:cE0-gamman1-alphan1-0}
\cE_{0,\kappa}^{\rm HFB} (\gamma_n^{(1)},\alpha_{n}^{(1)}) \to 0 \ \ \mbox{as}\ \, n\to \infty.
\end{align}

On the other hand, since $\gamma_n^{(1)} \to \gamma^{(1)}$ strongly in $\fS_{1}$ as $n\to\infty$, similar to \cite[Corollary 4.1]{LenLew-10}, we can prove  that
\begin{align}\label{eq:cE0-gamman1-alphan1-limit}
\liminf_{n\to \infty}\cE_{0,\kappa}^{\rm HFB} (\gamma_n^{(1)},\alpha_{n}^{(1)}) \ge \cE_{0,\kappa}^{\rm HFB} (\gamma^{(1)},\alpha^{(1)}) .
\end{align}
In fact, since $\Tr (\sqrt{-\Delta }\gamma_n^{(1)}) \lesssim 1$ holds for all $n$, we get from \eqref{3.9} that  $ \rho_{\gamma_n^{(1)}} \to \rho_{\gamma^{(1)}} $ strongly in $L^p(\R^3)$ as $n\to\infty$, where $1\le p<3/2$. Applying the Hardy--Littlewood--Sobolev inequality, this then implies that
\begin{align} \label{3.15}
D(\rho_{\gamma_n^{(1)}},\rho_{\gamma_n^{(1)}})\to D(\rho_{\gamma^{(1)}},\rho_{\gamma^{(1)}})\ \ \ \text{as}\ \ n\to\infty.
\end{align}
Moreover,  since $0\le \gamma_n^{(1)} - \alpha_n^{(1)} (\alpha_n^{(1)} )^* \rightharpoonup \gamma^{(1)} - \alpha^{(1)} (\alpha^{(1)} )^*$ weakly-* in $\fS_{1}$ as $n\to\infty$, by Fatou's lemma, we conclude that
\begin{align} \label{3.16}
\liminf_{n\to \infty}  \Tr\big(\sqrt{-\Delta} (\gamma_n^{(1)} - \alpha_n^{(1)} (\alpha_n^{(1)} )^*)\big)   \ge  \Tr\big(\sqrt{-\Delta} (\gamma^{(1)} - \alpha^{1} (\alpha^{(1)} )^*)\big).
\end{align}
Recall that $\alpha_n^{(1)}\rightharpoonup \alpha^{(1)}$ weakly in $\fS_{2}$ as $n\to\infty$, and it yields from the Hardy--Kato inequality that for $\kappa\in(0, 4/\pi)$,
$$
\sqrt{-\Delta_x} -  \frac{\kappa}{2|x-y|}\ge\sqrt{-\Delta_x} -  \frac{2}{\pi|x-y|}\ge 0\ \ \ \text{on}\ \ L^2(\R^3\times\R^3).
$$
By Fatou's lemma, this then implies  that  for $\kappa\in(0, 4/\pi)$,
\begin{equation} \label{eq:cE0-gamman1-alphan1-limit-b1}
\begin{split}
&\liminf_{n\to \infty} \Big\langle \alpha_n^{(1)}, \Big(\sqrt{-\Delta_x} -  \frac{\kappa}{2|x-y|}\Big) \alpha_n^{(1)}  \Big\rangle_{L^2(\R^3\times \R^3)}  \\
& \ge  \Big\langle \alpha^{(1)}, \Big(\sqrt{-\Delta_x} -  \frac{\kappa}{2|x-y|}\Big) \alpha^{(1)}  \Big\rangle_{L^2(\R^3\times \R^3)}.
\end{split}
\end{equation}
We thus deduce  from \eqref{3.16} and \eqref{eq:cE0-gamman1-alphan1-limit-b1} that
\begin{align}\label{eq:cE0-gamman1-alphan1-limit-b}
\liminf_{n\to \infty}\left( \Tr(\sqrt{-\Delta}\gamma_n^{(1)}) - \frac{\kappa}{2} \Ex(\alpha_n^{(1)}) \right) \ge \Tr(\sqrt{-\Delta}\gamma^{(1)}) -  \frac{\kappa}{2} \Ex(\alpha^{(1)}).
\end{align}
Since $(\gamma_n^{(1)}, \alpha_n^{(1)})\rightharpoonup (\gamma^{(1)}, \alpha^{(1)})$ weakly-* in $\cX$ as $n\to\infty$, we have
$$
\liminf_{n\to \infty} \Ex(\gamma_n^{(1)}) \ge \Ex(\gamma^{(1)}).$$
Therefore,   the claim \eqref{eq:cE0-gamman1-alphan1-limit} follows immediately from \eqref{3.15} and \eqref{eq:cE0-gamman1-alphan1-limit-b}.

Combining \eqref{eq:cE0-gamman1-alphan1-0} and \eqref{eq:cE0-gamman1-alphan1-limit},  we obtain that
\begin{align}\label{gyj}
(\gamma^{(1)},\alpha^{(1)}) \ \ \text{is a minimizer for}\ I^{\rm HFB}_{0,\kappa}(\lambda^1).
\end{align} Moreover, as explained in \cite[Proposition 4.1]{LenLew-10},  we also deduce from \eqref{eq:cE0-gamman1-alphan1-limit} that $(\gamma_n^{(1)},\alpha_n^{(1)}) $ $ \to (\gamma^{(1)},\alpha^{(1)})$ strongly in $\cX$ as $n\to\infty$.

Finally, if $0<\lambda^1<\lambda^{\rm HFB}_\kappa$, then we deduce from \eqref{gyj} and Lemma \ref{lem:HFB-alternative} that $I^{\rm HFB}_{m,\kappa}(\lambda^1)$ has no minimizer for every $m>0$. However, this contradicts with the existence result of  \cite[Theorem 1]{LenLew-10}.  Therefore, the dichotomy does not occur, which further yields that $\lambda^1=\lambda^{\rm HFB}_\kappa$. We hence conclude that $I^{\rm HFB}_{0,\kappa}(\lambda^{\rm HFB}_\kappa)$ has a minimizer $(\gamma^{(1)},\alpha^{(1)})$, and   every minimizing sequence of $I^{\rm HFB}_{0,\kappa}(\lambda^{\rm HFB}_\kappa)$ is relatively compact in $\cX$. Applying Lemma \ref{lem:HFB-alternative}, we also obtain that for every $m>0$,
$$I^{\rm HFB}_{m,\kappa}(\lambda^{\rm HFB}_\kappa)=-m \lambda^{\rm HFB}_\kappa,$$
but $I^{\rm HFB}_{m,\kappa}(\lambda^{\rm HFB}_\kappa)$ has no minimizer. The proof of Theorem \ref{thm:HFB-non-exist} is therefore complete.
\end{proof}

Applying Theorem \ref{thm:HFB-non-exist}, we are now ready to complete the proof of Theorem \ref{lem:GN-HFB}.

\begin{proof}[Proof of Theorem \ref{lem:GN-HFB}]  Let $\kappa\in(0, 4/\pi)$. Recall from Theorem \ref{thm:HFB-non-exist} that $I^{\rm HFB}_{0,\kappa}(\lambda^{\rm HFB}_\kappa)=0$ has a minimizer.
We then deduce that
\begin{align}\label{eq:kappa-dual-lambda}
\kappa = \inf \Big\{ \frac{2\Tr (\sqrt{-\Delta} \gamma)}{D(\rho_\gamma,\rho_\gamma)-X(\gamma)+X(\alpha)}:\,  (\gamma,\alpha)\in \cK,\,\Tr \gamma = \lambda^{\rm HFB}_\kappa\Big\},
\end{align}
and the above infimum can be attained.
In particular, this implies that $\kappa\mapsto \lambda^{\rm HFB}_\kappa$ is {\em injective}. Note from   \cite{LenLew-10} that  $\kappa\mapsto\lambda^{\rm HFB}_\kappa$ is nonincreasing and continuous. We thus conclude that $\kappa\mapsto \lambda^{\rm HFB}_\kappa$ is strictly decreasing,  and it is a bijective map from $(0,4/\pi)$ to $(\lambda^*,\infty)$, where $\lambda^*=\lim_{\kappa\to (4/\pi)-} \lambda^{\rm HFB}_\kappa$.

Consequently, for every $\lambda\in (\lambda^*,\infty)$, there exists a unique  $\kappa^{\rm HFB}_\lambda\in (0,4/\pi)$ satisfying the fixed-point constraint
\begin{align}\label{eq:kappa-lambda-fixed-point}
\lambda^{\rm HFB}_{ \kappa^{\rm HFB}_\lambda} =\lambda.
\end{align}
By the definition of the critical mass in \eqref{eq:Chandrasekhar-limit},  for every $\kappa\in(0, 4/\pi)$, we can  follow \eqref{scale-E0} to write
\begin{align}\label{eq:lambda-critical-v2}
\lambda^{\rm HFB}_\kappa &= \max \left\{ \lambda>0:\, I^{\rm HFB}_{m,\kappa} (\lambda) >-\infty \text{ for all } m\ge 0\right\} \nn\\
&=\max \left\{ \lambda>0:\,  I^{\rm HFB}_{0,\kappa}(\lambda) =0\right\}.
\end{align}

The characterization  \eqref{eq:kappa-critical-lambda} then follows directly from \eqref{eq:kappa-lambda-fixed-point} and \eqref{eq:lambda-critical-v2}. Moreover,   we deduce from \eqref{eq:kappa-dual-lambda} and \eqref{eq:kappa-lambda-fixed-point} that the identity \eqref{eq:GN-HFB} holds true and the minimization problem  \eqref{eq:GN-HFB} has a miminizer  for all $\lambda\in (\lambda^*,\infty)$. 
Finally, it yields from \eqref{eq:lambdaHFB} and \eqref{eq:kappa-lambda-fixed-point} that
\begin{align}\label{eq:lambda-critical-asymp-proof}
\tau_c^{3/2}=\lim_{\kappa\to 0^+} \kappa^{3/2} \lambda^{\rm HFB}_\kappa =  \lim_{\lambda\to \infty} (\kappa^{\rm HFB}_\lambda)^{3/2} \lambda^{\rm HFB}_{ \kappa^{\rm HFB}_\lambda} =  \lim_{\lambda\to \infty} (\kappa^{\rm HFB}_\lambda)^{3/2} \lambda,
\end{align}
$i.e.,$ $\kappa^{\rm HFB}_\lambda \sim \tau_c\lambda^{-2/3}$ as $\lambda\to \infty$. The proof of  Theorem \ref{lem:GN-HFB} is therefore complete.
\end{proof}

\section{Hartree--Fock minimizers }\label{sec:GN-ineq-ON}

From this section on, we focus on the minimizers of the Hartree-Fock (HF) theory. The main purpose of this section is to address Theorems \ref{th2.1} and \ref{cor1} on the existence and some qualitative properties of optimizers for $\kappa^{\rm HF}_N$. For this purpose, in this section we first prove the existence of  minimizers for the variational problem associated to \eqref{orthonormal}. We then establish in Section \ref{sec:43} the refined properties of these minimizers by deriving the strict inequality $\kappa_N^{\rm HF}<\kappa_{N-1}^{\rm HF}$. Theorems \ref{th2.1} and \ref{cor1} are finally proved in Section \ref{sec:44}.

To avoid the confusion, we  denote by $\kappa_N^*$ the optimal constant of the interpolation inequality \eqref{orthonormal} (we will prove that $\kappa_N^*=\kappa_N^{\rm HF}$ at the end). We start with the following variational problem associated with \eqref{orthonormal}.

\begin{lem}\label{than}
For every fixed $2\leq N\in\mathbb{N}$, the variational problem
\begin{equation}\label{an}
\begin{split}
\kappa^{*}_N:&=\inf\limits_{\gamma\in\mathcal{R}_{N}}
\frac{2\|\gamma\|\mathrm{Tr}\big(\sqrt{-\Delta}\, \gamma\big)}{D(\rho_{\gamma},\rho_{\gamma})-\Ex(\gamma)}
\end{split}
\end{equation}
has at least one optimizer, where $\mathcal{R}_{N}$ is defined in \eqref{rn}. Moreover, any optimizer $\gamma$ of \eqref{an} satisfying $\|\gamma\|=1$ commutes with the mean-field Hamiltonian
\begin{equation}\label{h81}
H^{*}_{\gamma}:=\sqrt{-\Delta}-\kappa^{*}_N\big(\rho_{\gamma}\ast|\cdot|^{-1}\big)(x) +\kappa^{*}_N\frac{\gamma(x,y)}{|x-y|} \ \ \ on \ \ L^2(\R^3).
\end{equation}
\end{lem}

\begin{proof} Since the proof follows from the standard method for calculus of variations, we only sketch its main points by two steps.

\medskip
\noindent	
{\bf Step 1.}
In this step, we prove that \eqref{an} has an optimizer. By scaling, one can choose a minimizing sequence $\{\gamma_{n}\}$ of \eqref{an} such that
\begin{equation}\label{a.7}
\mathrm{Tr}\big(\sqrt{-\Delta}\, \gamma_{n}\big)=\|\gamma_{n}\|=1,\ \   D(\rho_{\gamma_{n}},\rho_{\gamma_{n}})-\Ex(\gamma_{n}) = 2( \kappa^{*}_N)^{-1}>0.
\end{equation}
Since $\big\{\Tr ( \sqrt{-\Delta}\gamma_n)\big\}$ is bounded uniformly in $n>0$, up to a subsequence,  we have
$$(-\Delta)^{1/4}  \gamma_n (-\Delta)^{1/4} \rightharpoonup (-\Delta)^{1/4}  \gamma (-\Delta)^{1/4} \quad \text{ weakly-* in } \fS^1 \ \ \text{as} \ \ n\to\infty,
$$
where $\gamma \in \cR_N$ satisfies $\|\gamma\|\le 1$ and $\Tr (\sqrt{-\Delta}\gamma)\le 1$. Moreover, since $\{\sqrt{\rho_{\gamma_n}}\}$ is bounded uniformly in ${H}^{1/2}(\R^3)$, by Sobolev's embedding theorem we obtain that
$$
\rho_{\gamma_n}
\rightarrow \rho_\gamma\ \  \text{strongly\ in} \ L_{loc}^{p}(\R^3)\ \ \text{as}\ \ n\to\infty,  \ \ \forall\ 1\leq p<3/2.
$$
We now claim that the vanishing case of $\{\rho_{\gamma_{n}}\}$ does not occur, namely,
\begin{equation}\label{vani}
\liminf_{n\rightarrow \infty}\sup\limits_{z\in\R^3}\int_{B_1(z)} \rho_{\gamma_{n}}dx>0.
\end{equation}
Indeed, if \eqref{vani} were false, then the same argument of \eqref{eq:non-vanishing-HFB} gives that up to a subsequence,
$$0\leq \Ex(\gamma_{n}) \le D(\rho_{\gamma_{n}},\rho_{\gamma_{n}})\to 0\ \ \text{as}\ \ n\to\infty.
$$
This however contradicts with the assumption that $D(\rho_{\gamma_{n}},\rho_{\gamma_{n}})-\Ex(\gamma_{n}) = 2( \kappa^{*}_N)^{-1}>0$. The claim \eqref{vani} hence holds true.

As a consequence of the claim \eqref{vani},  up to a translation $\gamma_n(x,y)\mapsto \gamma_n(x-z_n,y-z_n)$ if necessary, we have $\gamma\ne 0$.
By the strongly local convergence (cf. \cite[Lemma 7.3]{LenLew-10}),   there exists a  sequence $\{R_n\}$ satisfying  $R_n\rightarrow\infty$ as $n\to\infty$  such that up to a subsequence  if necessary,
\begin{equation}\label{2.13}
0<\int_{\R^3}\rho_\gamma dx = \lim\limits_{n\to\infty}\int_{|x|\leq R_n}\rho_{{\gamma}_n}dx,\quad \lim\limits_{n\to\infty}\int_{R_n\leq|x|\leq 2R_n}\rho_{{\gamma}_n}dx=0.
\end{equation}
Let $\chi_{R_n}^2+\eta_{R_n}^2=1$ be as in \eqref{eq:def-chiR-etaR},  and define $\gamma_n^{(1)}$ and $\gamma_n^{(2)}$ to be  as in \eqref{eq:gamman1-gamman2}. It  follows from \eqref{2.13} that  $\gamma_n^{(1)}\to \gamma$ strongly in $\fS^1$ as $n\to\infty$, and thus $\|\rho_{\gamma_n^{(1)}}-\rho_\gamma\|_{L^p}\to 0$ as $n\to\infty$ for all $1\le p<3/2$. Using the IMS localization formulas  and the Hardy--Littlewood--Sobolev inequality,  the same arguments of \eqref{eq:HFB-splitting-E} and  \eqref{3.15} then give that
\begin{align}\label{2.20}
\mathrm{Tr}\big(\sqrt{-\Delta}\, {\gamma}_n\big)
&= \mathrm{Tr}\big(\sqrt{-\Delta}\, \gamma_n^{(1)}\big)+\mathrm{Tr}\big(\sqrt{-\Delta}\, \gamma_n^{(2)}\big)
+o(1)  \nn\\
&\ge \mathrm{Tr}\big(\sqrt{-\Delta}\, {\gamma}\big)+\mathrm{Tr}\big(\sqrt{-\Delta}\, \gamma_n^{(2)}\big)
+o(1)\ \ \text{as}\ \ n\to\infty, \\
\Ex(\gamma_n) 
&\ge  \Ex(\gamma_n^{(1)})+  \Ex(\gamma_n^{(2)}) \ge \Ex(\gamma)+  \Ex(\gamma_n^{(2)}) + o(1)\ \ \text{as}\ \ n\to\infty,\nn\\
D(\rho_{\gamma_n},\rho_{\gamma_n}) 
&= D(\rho_{\gamma_n^{(1)}},\rho_{\gamma_n^{(1)}}) +  D(\rho_{\gamma_n^{(2)}},\rho_{\gamma_n^{(2)}}) + o(1) \nn\\
&=  D(\rho_{\gamma},\rho_{\gamma}) + D(\rho_{\gamma_n^{(2)}},\rho_{\gamma_n^{(2)}})  + o(1)\ \ \text{as}\ \ n\to\infty,\nn
\end{align}
where the first identity of $D(\rho_{\gamma_n},\rho_{\gamma_n})$ follows from the fact that 
$$
\|\rho_{{\gamma}_{n}}\|_{L^{p}(B_{2R_n}\backslash B_{R_n})}\to 0\ \  \ \text{as}\ \ n\to\infty,\ \ \forall\ 1\le p<3/2.
$$

We now conclude from \eqref{an}, \eqref{a.7} and \eqref{2.20}  that
\begin{align*}
1&=\ \mathrm{Tr}\big(\sqrt{-\Delta}\, {\gamma}_n\big) \geq\ \mathrm{Tr}\big(\sqrt{-\Delta}\, \gamma\big)+\mathrm{Tr}\big(\sqrt{-\Delta}\, \gamma_n^{(2)}\big)
+o(1)\\
&\geq\ \|\gamma\|\, \mathrm{Tr}\big(\sqrt{-\Delta}\, \gamma\big)+\|{\gamma}_{n}^{(2)}\|\, \mathrm{Tr}\big(\sqrt{-\Delta}\, {\gamma}_{n}^{(2)}\big)
+o(1)\\
&\geq\ \frac{\kappa_N^*}{2} \Big(D(\rho_\gamma,\rho_\gamma)-\Ex(\gamma)\Big) +\frac{\kappa_N^*}{2} \Big(D(\rho_{{\gamma}_{n}^{(2)}}, \rho_{{\gamma}_{n}^{(2)}})-\Ex({\gamma}_{n}^{(2)}) \Big)	+o(1)\\
&\geq \frac{\kappa^{*}_N}{2}\Big(D(\rho_{\gamma_n}, \rho_{\gamma_n})
-\Ex(\gamma_n) \Big)+o(1) = 1+ o(1)\ \ \, \text{as}\ \ n\to \infty,
\end{align*}
where we  have used the facts that
$$\max\big\{\|\gamma_n^{(1)}\|, \|\gamma_n^{(2)}\| \big\}  \le  \|{\gamma}_n\| =1\ \ \mbox{and} \ \ \|\gamma\|=\lim_{n\to\infty} \|\gamma_n^{(1)}\|\le 1.$$
Consequently, we obtain that
$$ \|\gamma\|\, \mathrm{Tr}\big(\sqrt{-\Delta}\, \gamma\big)=\frac{\kappa_N^*}{2} \Big[D(\rho_\gamma,\rho_\gamma)-\Ex(\gamma)\Big].$$
This therefore yields that $\gamma$ is an optimizer of the variational problem \eqref{an}, and Step 1 is proved.

\medskip
\noindent	
{\bf Step 2.} In this step, we prove that $\gamma $ commutes with the mean-field operator $H^{*}_{\gamma }$ of \eqref{h81}. To address it, replacing $\gamma$ by  $\gamma/\|\gamma\|$ if necessary, one can suppose that $\gamma$ is an optimizer of \eqref{an} satisfying $\|\gamma\|=1$. Define
$$
\gamma(t):=e^{itA}\gamma e^{-itA},\ \ \text{where}\ t\in\R,\ A=|\varphi\rangle\langle\varphi| \ \text{and}\ \varphi\in C_{0}^\infty(\R^3),
$$
so that $\|\gamma(t)\|=\|\gamma\|=1$ and $\gamma(t)\in\mathcal{R}_{N}$. Note that
$$\gamma(t)=\gamma+tB+o(t)\ \ \text{as}\ \  t\to0, \ \ B=i\big(A\gamma-\gamma A\big).$$
Since $\gamma$ is an optimizer of \eqref{an}, we can further get that
$$
\frac{d}{dt} \Big(\|\gamma(t)\|\mathrm{Tr}\big(\sqrt{-\Delta}\,  \gamma(t)\big) - \frac{\kappa_N^*}{2} \Big[ D\big(\rho_{\gamma(t)},\rho_{\gamma(t)}\big)-\Ex(\gamma(t)) \Big] \Big)\Big|_{t=0} =0,
$$
which implies that
\begin{align}\label{2.22a}
\mathrm{Tr}\big(H^{*}_{\gamma}\,  B\big)=i\kappa^{*}_N \Im \int_{\R^{3}}\int_{\R^{3}}\frac{\gamma(x,y)\, \overline{B(x,y)}}{|x-y|}dxdy.
\end{align}
Moreover, since $\gamma, A$ and $B$ are self-adjoint, one can check that
\begin{align*}
&\Im \int_{\R^{3}}\int_{\R^{3}}\frac{\gamma(x,y)\, \overline{B(x,y)}}{|x-y|}dxdy\\
&=\text{Re} \int_{\R^{3}}\int_{\R^{3}}\frac{\gamma(x,y)\big(A\gamma-\gamma A\big)(y,x)}{|x-y|}dxdy\\
&=\text{Re}\int_{\R^{3}}\int_{\R^{3}}\frac{\overline{\gamma(y,x)}\cdot  \overline{\big(\gamma A\big)(x,y)}}{|x-y|}dxdy-\text{Re}\int_{\R^{3}}\int_{\R^{3}}\frac{\gamma(x,y)\big(\gamma A\big)(y, x)}{|x-y|}dxdy\nonumber\\
&=\text{Re}\int_{\R^{3}}\int_{\R^{3}}\frac{\overline{\gamma( x, y)\big(\gamma A\big)(y, x)}}{|x-y|}dxdy-\text{Re}\int_{\R^{3}}\int_{\R^{3}}\frac{\gamma(x,y)\big(\gamma A\big)(y, x)}{|x-y|}dxdy=0.
\end{align*}
This further gives from \eqref{2.22a} that $\mathrm{Tr}\big(H^{*}_{\gamma}\,  B\big)=0$.  Recalling  $A=|\varphi\rangle\langle\varphi|$, we thus deduce that
$$
\Big\langle \varphi,\  i\big(\gamma H^{*}_{\gamma}-H^{*}_{\gamma}\gamma\big)\varphi\Big\rangle=0,\quad \forall \ \varphi\in C_c^\infty(\R^3).
$$
Since $i\big(\gamma H_{\gamma}-H_{\gamma}\gamma\big)$ is self-adjoint, we obtain $[H^{*}_{\gamma},\gamma]=0$. The proof of
Lemma \ref{than} is therefore complete.
\end{proof}

\linespread{1.0}

\subsection{From strict monotonicity to projection property} \label{sec:43} 
In this subsection, we shall prove that any optimizer $\gamma $ of \eqref{an} satisfying $\|\gamma \|=1$ must be a rank-$N$ projection under the following strict inequality
$$\kappa_N^* < \kappa_{N-1}^{*}.$$
We then prove that the above strict inequality  naturally holds true. More precisely, in this subsection we shall prove the following key property.

\begin{lem}\label{than1} Set $\kappa_1^*:=\infty$, and let $\kappa_N^*>0$ be defined by \eqref{an} for $2\leq N\in\mathbb{N}$. If $\kappa_N^*<\kappa_{N-1}^*$, then
any optimizer $\gamma $ of \eqref{an} satisfying $\|\gamma \|=1$ can be written as $\gamma =\sum_{j=1}^{N}|w_{j}\rangle\langle w_{j}|$, where the orthonormal functions $\{w_j\}_{j=1}^N$ are eigenfunctions of $H_{\gamma }^*$ in \eqref{h81} associated to the eigenvalues $\nu_{1}\le \cdots \le \nu_N<0$.
\end{lem}

\linespread{-0.01}

\begin{proof}
Let $\gamma $ be an optimizer of \eqref{an} satisfying $\|\gamma \|=1$. The assumption condition $\kappa^{*}_N<\kappa_{N-1}^{*}$ implies that ${\rm Rank} \gamma =N$. Following the fact that $[H^{*}_{\gamma },\gamma ]=0$, where $H^{*}_{\gamma }$ is defined by (\ref{h81}), then we can write
\begin{equation}\label{2.21a}
\gamma =\sum_{j=1}^{{N}}\mu_{j}|w_{j}\rangle\langle w_{j}|,\ \ \mu_{j}\in(0, 1],\ \  \langle w_{j}, w_{k}\rangle=\delta_{jk},
\end{equation}
where  $w_{1}, \cdots , w_N$ are eigenfunctions of $H^{*}_{\gamma }$ associated to the eigenvalues $\nu_{1}\le \cdots \le \nu^{*}_{{N}}$.  Following (\ref{2.21a}), the rest proof of Lemma \ref{than1} is to prove that both $\mu_{j}=1$ and $\nu^{*}_{j}<0$ hold for all $j= 1,\cdots, N$.

For every non-negative self-adjoint operator $(\gamma'-\gamma)$, we have
$$\Ex (\gamma'-\gamma) \le D(\rho_{\gamma'-\gamma},\rho_{\gamma'-\gamma}),$$
and hence
\begin{align}\label{2.35cb}
&D(\rho_{\gamma'},\rho_{\gamma'})-\Ex(\gamma')\nonumber\\
&=\, D(\rho_{\gamma },\rho_{\gamma })-\Ex\big(\gamma \big)+2\Tr \Big[ \Big( \rho_\gamma* |x|^{-1} - \frac{\gamma(x,y)}{|x-y|}\Big) (\gamma'-\gamma) \Big] \nn\\
&\quad + D(\rho_{\gamma'-\gamma},\rho_{\gamma'-\gamma})-\Ex(\gamma'-\gamma)\\
&=\, 2(\kappa^{*}_N)^{-1}\mathrm{Tr}\big(\sqrt{-\Delta}\, \gamma \big)+ 2(\kappa^{*}_N)^{-1} \left[ \Tr \big(\sqrt{-\Delta}(\gamma'-\gamma)\big) - \Tr \big(H_{\gamma }^* (\gamma'-\gamma)\big) \right]\nn\\
&\quad + D(\rho_{\gamma'-\gamma},\rho_{\gamma'-\gamma})-\Ex(\gamma'-\gamma)\nn\\
&\ge 2(\kappa^{*}_N)^{-1}\left[ \mathrm{Tr}\big(\sqrt{-\Delta}\,\gamma' \big) - \Tr \big(H_{\gamma }^* (\gamma'-\gamma)\big)  \right].\nn
\end{align}
Define
\begin{align*}\label{2.33a}
\gamma':=\gamma -\mu_{j}|w_{j}\rangle\langle w_{j}|,\ \ \|\gamma'\|\le 1,\ \ {\rm Rank} (\gamma')=N-1.
\end{align*}
We now claim that 
\begin{equation}\label{4:2.33a}
\nu_j<0\ \  \mbox{holds for all }\ j= 1,\cdots, N.
\end{equation}  
Indeed, if $N=2$, then $D(\rho_{\gamma'},\rho_{\gamma'})-\Ex(\gamma')=0$. This implies from \eqref{2.35cb} that
$$
0 \ge \mathrm{Tr}\big(\sqrt{-\Delta}\,\gamma' \big)  - \Tr \big(H_{\gamma }^* (\gamma'-\gamma )\big) = \mathrm{Tr}\big(\sqrt{-\Delta}\,\gamma' \big)  +\nu_j \mu_j,
$$
and hence
$$\nu_j\le - \mu_j^{-1}\mathrm{Tr}\big(\sqrt{-\Delta}\,\gamma' \big) <0,\ \  j= 1, 2.$$
If $N\ge 3$, then following the assumption $\kappa^{*}_N <\kappa^{*}_{N-1}$, we deduce from  \eqref{2.35cb} that
\begin{align*} 
0<\kappa^{*}_N <\kappa^{*}_{N-1}
\leq\frac{2\|\gamma'\|\, \mathrm{Tr}\big(\sqrt{-\Delta}\, \gamma'\big)}{D(\rho_{\gamma'},\rho_{\gamma'})-\Ex(\gamma')}
\le \kappa^{*}_N\frac{\mathrm{Tr}\big(\sqrt{-\Delta}\,  \gamma'\big)}
{\mathrm{Tr}\big(\sqrt{-\Delta}\,  \gamma'\big)+\nu_j \mu_j } ,
\end{align*}
which also implies that $\nu_j<0$ holds for all $j= 1,\cdots, N$. This proves the claim (\ref{4:2.33a}).

On the contrary, assume that the coefficient $\mu_j<1$ of (\ref{2.21a}) holds for some $j\in \{1,\cdots,N\}$. Applying  \eqref{2.35cb} with
$$\gamma': =\gamma +\varepsilon|w_{j}\rangle\langle w_{j}|,\ \ 0<\varepsilon<1-\mu_j,$$
we then obtain that
\begin{equation*}
\begin{split}
\kappa^{*}_N&\leq\frac{2\mathrm{Tr}(\sqrt{-\Delta} \gamma')}{D(\rho_{\gamma'},\rho_{\gamma'})-\Ex(\gamma')}\le \kappa^{*}_N\frac{\mathrm{Tr}\big(\sqrt{-\Delta} \gamma'\big)}
{\mathrm{Tr}\big(\sqrt{-\Delta} \gamma'\big)-\varepsilon\nu_{j}}<\kappa^{*}_N,
\end{split}
\end{equation*}
where the last inequality  follows from the fact that $\nu_j<0$  holds for all $j= 1,\cdots, N$.  This is however a contradiction, and hence $\mu_{j}=1$ holds for all $j=1,\cdots,N$. This therefore completes the proof of Lemma \ref{than1}.
\end{proof}

Inspired by the existence-nonexistence alternative of Lemma \ref{lem:HFB-alternative}, the following lemma illustrates that the  strict inequality $\kappa_{N-1}^{*}>\kappa^{*}_N$ of  Lemma \ref{than1} holds in a natural way.

\vspace{3mm}
\begin{lem}\label{nomin}
For any fixed $m>0$ and  $ 2\leq N\in\mathbb{N}$,  let $I^{\rm HF}_{m,\kappa^{\rm HF}_N}(N)$  and  $\kappa^{*}_N>0$ be defined by \eqref{problem} and \eqref{an}, respectively, where we denote $\kappa_1^*=\infty$. Then $\kappa_{N-1}^{*}>\kappa^{*}_N$ and $I^{\rm HF}_{m,\kappa^{*}_N}(N)=-mN$, but $I^{\rm HF}_{m,\kappa^{*}_N}(N)$ has no minimizer.
\end{lem}
\linespread{0.7}
\begin{proof}
By induction on $N$, we first  prove that the case $N=2$ holds true. In this case, we always have $\kappa_2^*<\kappa_1^*=\infty$. This then implies from Lemma \ref{than1} that the variational problem $\kappa_2^*$ of \eqref{an}  has a minimizer $\gamma_2=\sum_{j=1}^{2}|w_{j}\rangle\langle w_{j}|$,  where $\langle w_{j},w_{k}\rangle=\delta_{jk}$ holds for $j, k=1,2$. 
Denote $\gamma_{2n}(x,y):=n^3 \gamma_2(nx,ny)$ for any $n>0$. One  can then verify  that
$$
0 = I^{\rm HF}_{0,\kappa_2^*}(2) =\mathcal{E}^{\rm HF}_{0,\kappa_2^*}(\gamma_{2})=\frac{1}{n} \mathcal{E}^{\rm HF}_{0,\kappa_2^*}(\gamma_{2n}).
$$
By the Dominated Convergence Theorem,  we further conclude that
\begin{equation}\label{2.32}
\begin{split}
0 &\leq  I^{\rm HF}_{m, \kappa^*_{2}}(2) + 2m \leq \mathcal{E}^{\rm HF}_{m,\kappa_2^*}(\gamma_{2n}) + 2m\\[2mm] &=\Tr\Big[\big(\sqrt{-\Delta+m^{2}}-\sqrt{-\Delta}\, \big) \gamma_{2n}\Big] \\
&=\sum_{j=1}^{2}\int_{\R^3}\big(\sqrt{n^{2}|\xi|^{2}+m^{2}}-\sqrt{n^{2}|\xi|^{2}}\, \big)|\widehat{w}_{j}(\xi)|^{2} d\xi\\
&= \sum_{j=1}^{2}\int_{\R^3} \frac{|\xi|^2}{ \sqrt{n^{2}|\xi|^{2}+m^{2}}+ \sqrt{n^{2}|\xi|^{2}}} |\widehat{w}_{j}(\xi)|^{2} d\xi \to 0\ \ \mbox{as}\,\ n\to \infty,
\end{split}\end{equation}
which  yields that $I^{\rm HF}_{m,\kappa_{2}^{*}}(2)=-2m$. On the other hand,  since $\mathcal{E}^{\rm HF}_{0,\kappa_2^*}(\gamma) \ge I^{\rm HF}_{0,\kappa_2^*}(2)=0$ for every $\gamma\in \cP_2$,  we deduce that 
$$
\mathcal{E}^{\rm HF}_{m,\kappa_2^*}(\gamma) \ge \mathrm{Tr}\big(\sqrt{-\Delta+m^{2}}\, \gamma\big) - \mathrm{Tr}\big(\sqrt{-\Delta}\, \gamma\big) -2m>-2m,\,  \ \forall \ \gamma\in \cP_2.
$$
Therefore, $I^{\rm HF}_{m,\kappa_{2}^{*}}(2)$ has no minimizer for any $m>0$. Therefore,  the case $N=2$ holds true.

Since $I^{\rm HF}_{m,\kappa_{2}^{*}}(2)$ has no minimizer, we deduce from \cite[Theorem 4]{LenLew-10} that
\begin{equation}\label{2.33}
I^{\rm HF}_{m,\kappa_{2}^{*}}(3)=-\infty.
\end{equation}
On the other hand, the definition of $\kappa_{3}^{*}$ in \eqref{an}  gives that
$I^{\rm HF}_{m,\kappa_{3}^{*}}(3)\geq -3m>-\infty$. This then   implies from  \eqref{2.33} that
\begin{equation}\label{2.34}
\kappa_{2}^{*}>\kappa_{3}^{*}.
\end{equation}
As a consequence of \eqref{2.34}, we obtain from Lemma \ref{than1}  that $\kappa_3^*$ admits an optimizer of the form  $\gamma_3=\sum_{j=1}^{3}|w_{j}\rangle\langle w_{j}|$, where $\langle w_{j},w_{k}\rangle=\delta_{jk}$. Following the above argument of the case $N=2$, we then get that
$$
I^{\rm HF}_{m,\kappa_{3}^{*}}(3)=-3m,  \ \ \text{but\ $I^{\rm HF}_{m,\kappa_{3}^{*}}(3)$\ has\ no\ minimizer},
$$
and hence  the case $N=3$ also holds true.

Repeating the above procedure, by induction, one can obtain that for all $2\leq N\in\mathbb{N}$,
$$
\kappa_{N-1}^{*}>\kappa^{*}_N,\ \ \  I^{\rm HF}_{m,\kappa^{*}_N}(N)=-mN\  \ \text{has no minimizer}.$$
This therefore completes the proof of Lemma \ref{nomin}.
\end{proof}

\subsection{Proofs of Theorems \ref{th2.1} and \ref{cor1}} \label{sec:44}
This subsection is devoted to the proofs of Theorems \ref{th2.1} and \ref{cor1}. We first address Theorem \ref{th2.1} as follows.

\linespread{0.5}

\begin{proof}[Proof of Theorem \ref{th2.1}]  We shall carry out the proof by the following three steps:
	
\medskip
\noindent	
{\bf Step 1.} In this step, we prove Theorem  \ref{th2.1} (1), i.e., the sequence $\{\kappa^{\rm HF}_N\}_{N\ge 2}$ is strictly decreasing, and  $\kappa^{\rm HF}_N\sim \tau_c N^{-2/3}$ as $N\to \infty$, where $\tau_c>0$ is given in \eqref{eq:def-tauc}.

As a consequence of Lemmas \ref{than1} and \ref{nomin}, we can obtain that any optimizer $\gamma $ of $\kappa_N^*$ in \eqref{an} satisfying $\big\|\gamma \big\|=1$ must be an $N$-dimensional orthogonal projection. This thus implies that
\begin{equation}\label{4.16a}
\kappa_N^*=\kappa_N^{\rm HF},\ \  N\mapsto \kappa_N^{\rm HF}\ \text{is strictly decreasing},
\end{equation}
where $\kappa_N^{\rm HF}$ is defined in \eqref{eq:GN-HF}. Denote by $N^{\rm HF}_\kappa$ the largest integer such that $I^{\rm HF}_{m,\kappa}(N) > -\infty$. We then obtain from \eqref{4.16a} that the following fixed-point constraint holds:
$$
N=N^{\rm HF}_{\kappa^{\rm HF}_N},\ \ \ \forall \  2\le N\in \mathbb{N}.
$$
In particular, since $N^{\rm HF}_{\kappa}\to \infty$ as $\kappa\searrow 0$, we conclude that $\kappa_N^{\rm HF}\searrow 0$ as $N\to \infty$.

We now claim that
\begin{align}\label{3.67a}
\lim\limits_{\kappa\rightarrow 0^{+}}\kappa^{3/2}N^{\rm HF}_\kappa=\tau_{c}^{3/2}>0,
\end{align}
where $\tau_{c}\simeq 2.677$ is defined in \eqref{eq:def-tauc}. To prove \eqref{3.67a}, stimulated by  \cite[Appendix B]{LenLew-10}, we  consider the energy functional
$$
\cE^{\rm red}_{\kappa}(\gamma)=\Tr\big(\sqrt{-\Delta}\, \gamma\big)-\frac{\kappa}{2}D(\rho_\gamma, \rho_\gamma),\ \   \kappa\in (0,4/\pi).
$$
Since the map $f\mapsto D(f, f)$ is strictly convex, the same argument of \cite[Lemma 11]{ii} gives that for all $2\leq N\in\mathbb{N}$,
\begin{equation}\label{3.62}
I^{\rm red}_{\kappa}(N):= \inf\Big\{\cE_{\kappa}^{\rm red}(\gamma):\, 0\leq\gamma\leq1,\ \Tr(\gamma)=N\Big\}=\inf\limits_{\gamma\in\cP_{N}}\cE^{\rm red}_{\kappa}(\gamma),
\end{equation}
where $\cP_{N}$ is defined in \eqref{set}.
Denote by $N^{\rm red}_\kappa$ the critical mass associated with $I^{\rm red}_{\kappa}(N)$, i.e.,
\begin{equation*}
N^{\rm red}_\kappa=\max\big\{N\in\mathbb{N}:\,  I^{\rm red}_{\kappa}(N)>-\infty\big\}.
\end{equation*}
Following the analysis of  \cite[(B.3) and (B.8)]{LenLew-10}, it yields from \eqref{3.62} that the critical mass $N^{\rm HF}_\kappa\in\mathbb{N}$  of the HF problem satisfies
\begin{equation}\label{3.64}
N^{\text{HF}}_{\frac{\kappa}{1-\kappa\pi/4}}\leq N^{\rm red}_\kappa \leq N^{\rm HF}_\kappa\leq N^{\rm red}_{\frac{\kappa}{1+\kappa\pi/4}}.
\end{equation}
Note from \cite[Proposition 2.1 and Appendix B]{LenLew-10} that
$$
\lim\limits_{\kappa\rightarrow 0^{+}}\kappa^{3/2}N^{\rm red}_\kappa = \tau_{c}^{3/2}.
$$
Therefore, we can deduce from \eqref{3.64} that  \eqref{3.67a} holds true.

Similar to \eqref{eq:lambda-critical-asymp-proof}, we further obtain from  \eqref{3.67a} the following asymptotic behaviors of $\kappa_N^{\rm HF}$ and  the fixed-point constraint $m=N^{\rm HF}_{\kappa^{\rm HF}_m}$:
$$
\tau_c^{3/2}=\lim\limits_{\kappa\rightarrow 0^{+}}\kappa^{3/2}N^{\rm HF}_\kappa = \lim\limits_{N\to \infty} (\kappa_N^{\rm HF}) ^{3/2}N^{\rm HF}_{\kappa_N^{\rm HF}}=  \lim\limits_{N\to \infty} (\kappa_N^{\rm HF}) ^{3/2}N.
$$
This completes the proof of Step 1, and  the proof of Theorem  \ref{th2.1} (1) is therefore complete.

 \medskip
\noindent	
{\bf Step 2.} 
Applying Lemmas \ref{than}-\ref{nomin}, we can  suppose that
$\gamma =\sum_{j=1}^{N}|w_{j}\rangle\langle w_{j}|$ is  an optimizer of  $\kappa_N^{\rm HF}$, where the orthonormal  functions $\{w_{1}, \cdots, w_N\}$ are eigenfunctions of the mean-field operator $H_{\gamma }$ in \eqref{h8} associated to the negative eigenvalues $\nu_{1}\leq\nu_{2}\leq\cdots \leq \nu_{N}<0$. In this step, we shall prove that
the  regularity $w_{j}\in C^\infty(\R^3)$ holds for all $j=1,\cdots,N$.  By Sobolev's embedding theorem,  it   suffices  to prove that
\begin{align}\label{hs}
w_{j}\in H^s(\R^3)\ \ \text{for\ any}\ s\geq1,\ \ j=1, \cdots, N.
\end{align}

To prove (\ref{hs}), we first note that $w_j$ satisfies the following eigenvalue problem
\begin{equation}\label{2.37}
\begin{split}
\sqrt{-\Delta}\, w_{j} &=\nu_{j}w_{j}+\kappa^{\rm HF}_N\sum_{k=1}^N\int_{\R^{3}}\frac{|w_{k}(y)|^{2}}{|x-y|}dy\, w_{j}-\kappa^{\rm HF}_N\sum_{k=1}^N\int_{\R^{3}}\frac{\overline{w_{k}(y)}\, w_{j}(y)}{|x-y|}dy\, w_{k} \\
&:=h_{j}\ \ \text{in}\, \ \R^3,\ \ j=1,\cdots, N.
\end{split}
\end{equation}
It then yields that
\begin{equation}\label{2.37a}
\widehat{w}_{j}(\xi)=|\xi|^{-1}\widehat{h}_{j}(\xi)\ \ \text{in}\ \ \R^3,\ \ j=1,\cdots, N.
\end{equation}
Applying the   inequality (cf. \cite[Lemma 3.2]{optimal}) 
\begin{equation}\label{2.39a}
\Big\|\Big(|x|^{-1}\ast(\varphi_1\varphi_{2})\Big)\varphi_3\Big\|_{H^s}\leq C\prod_{k=1}^3\|\varphi_k\|_{H^s}, \ \ \forall\ s\geq\frac{1}{2},
\end{equation}
since $w_{j}\in H^{1/2}(\R^3)$, we get that the function $h_{j}$ of \eqref{2.37} satisfies
\begin{equation*}\label{2.40a}
h_{j}\in H^{1/2}(\R^3),\ \ j=1, \cdots, N,
\end{equation*}
where $C>0$ is independent of $\varphi_k$. We then conclude from \eqref{2.37a} that
\begin{equation}\label{2.42}	 \int_{\R^{3}}|\xi|^{2}|\widehat{w}_{j}(\xi)|^{2}d\xi=\int_{\R^{3}}|\widehat{h}_{j}(\xi)|^{2}d\xi=\|h_{j}\|_{2}^{2}<\infty,\ \ j=1, \cdots, N.
\end{equation}
Note that $(w_1, \cdots, w_{N})\in \big(H^s(\R^3)\big)^N$ holds for $s>0$, if and only if
\begin{equation}\label{2.38a}
\int_{\R^{3}}\big(1+|\xi|^{2s}\big)|\widehat{w}_{j}(\xi)|^{2}d\xi<\infty,\ \  j=1, \cdots, N.
\end{equation}
We thus deduce from \eqref{2.42} that
\begin{equation}\label{2.38b}	
(w_1, \cdots, w_{N})\in \big(H^s(\R^3)\big)^N\ \ \text{holds\ for}\ s=1,
\end{equation}
which further implies from \eqref{2.37} and \eqref{2.39a}
that $(h_1, \cdots, h_{N})\in \big(H^s(\R^3)\big)^N$ holds for $s=1$. The same argument of \eqref{2.38b} then gives from \eqref{2.37a} and \eqref{2.38a} that
\begin{equation*}\label{2.38c}	
(w_1, \cdots, w_{N})\in \big(H^s(\R^3)\big)^N\ \ \text{holds\ for}\ s=2.
\end{equation*}
By induction, we therefore conclude from above that $(w_1, \cdots, w_{N})\in \big(H^s(\R^3)\big)^N$ for any $s\geq1$, and hence \eqref{hs} holds true. 

 \medskip
\noindent	
{\bf Step 3.}
In this step, we prove the decaying property \eqref{decay}, which is equivalent to
\begin{equation}\label{decay1}
\sum_{j=1}^{N}|w_{j}(x)|\leq C\big(1+|x|\big)^{-4}\ \ \text{uniformly\ in}\,\ \R^3.
\end{equation}

To address the above decaying property \eqref{decay1}, we first claim that
\begin{equation}\label{2.39}
w_{j}(x)\to0\ \  \mathrm{as}\ \ |x|\to\infty,\ \ j=1,\cdots,N.
\end{equation}
Let $G_{j}(x)$ be the Green's function of the operator $\sqrt{-\Delta}-\nu_{j}$ in $\R^3$, where $\nu_{j}<0$ is as in \eqref{2.37}.
It then follows from \eqref{2.37} that
\begin{equation}\label{2.46b}
w_{j}(x)=\big(G_{j}\ast f_{j}\big)(x)\ \ \, \text{in}\ \ \R^3, \ \ j=1,\cdots,N,
\end{equation}
where
$$
f_{j}=\kappa^{\rm HF}_N\sum_{k=1}^N\int_{\R^{3}}\frac{|w_{k}(y)|^{2}}{|x-y|}dy\, w_{j}-\kappa^{\rm HF}_N\sum_{k=1}^N\int_{\R^{3}}\frac{\overline{w_{k}(y)}\, w_{j}(y)}{|x-y|}dy\, w_{k}.
$$
Since $w_{j}\in H^1(\R^3)$ holds for $j=1, \cdots, N$, we conclude from \eqref{2.39a} that $f_{j}\in H^1(\R^3)$, and thus
\begin{equation}\label{2.36}	
f_{j}\in L^p(\R^3),\  \  \forall\  p\in[2,6], \ \ j=1, \cdots, N.
\end{equation}
Note from \cite[Lemma C.1]{green} that
\begin{equation*}\label{green}
G_{j}\in L^r(\R^3, \R)\ \ \text{for\ any}\,\ r\in[1,\, 3/2),\ \ j=1, \cdots, N,
\end{equation*}
which then gives that
\begin{equation}\label{2.38}
G_{j}\in L^{\frac{p}{p-1}}(\R^3, \R)\ \ \text{for\ some}\,\ p\in(3,6],\ \ j=1, \cdots, N.
\end{equation}
Applying \eqref{2.46b}--\eqref{2.38}, we therefore conclude from \cite[Lemma A.2]{lcc} that the claim \eqref{2.39} holds true.

We finally prove that the decaying property \eqref{decay1} holds true. Actually, recall from \cite[(C.8)]{green} the general Kato-type inequality
\begin{align*}\label{Kato-ineq}
\sqrt{-\Delta}\, |\varphi|\leq \text{sgn}(\varphi) \sqrt{-\Delta}\, \varphi\ \ \text{a.e.\ in}\ \R^3,\ \ \forall\ \varphi\in H^{1/2}(\R^3),
\end{align*}
and note that
\begin{equation*}\label{2.41}
\begin{split}
&\sum_{k\neq j}^N\int_{\R^{3}}\frac{\overline{w_{k}(y)}w_{j}(y)}{|x-y|}dy\, w_{k}(x)\, \text{sgn}\big(w_{j}(x)\big)\\
\geq &-\int_{\R^{3}}\frac{\rho_{\gamma }(y)}{|x-y|}dy\sum_{k\neq j}^N|w_{k}(x)|\ \ \text{in}\, \ \R^3, \ \ j=1, \cdots, N,
\end{split}
\end{equation*}
where $\text{sgn}\varphi=\overline{\varphi}/|\varphi|$ holds for $\varphi(x)\neq0$, and $\text{sgn}\varphi=0$ otherwise. We then deduce from \eqref{2.37} that for $ j=1,\cdots, N$,
\begin{equation}\label{2.36a}
\begin{split}
\sqrt{-\Delta}\, |w_{j}|-\nu_{j} |w_{j}|-\kappa^{\rm HF}_N\int_{\R^{3}}\frac{\rho_{\gamma }(y)}{|x-y|}dy\, \sum_{k=1}^N|w_{k}|\leq0\ \ \text{in}\, \ \R^3.
\end{split}
\end{equation}
Since $\infty>-\nu_1\geq\cdots\geq-\nu_{N}>0$,
we further obtain from \eqref{2.36a} that
\begin{equation*}
\begin{split}
\sqrt{-\Delta}\, |w_{j}|-\nu_{N} |w_{j}|-\kappa^{\rm HF}_N\int_{\R^{3}}\frac{\rho_{\gamma }(y)}{|x-y|}dy\sum_{k=1}^N|w_{k}|\leq0\ \ \text{in}\, \ \R^3, \ \ j=1,\cdots, N.
\end{split}
\end{equation*}
Denoting $W:=\sum_{k=1}^N|w_{k}|$, we thus deduce that
\begin{equation}\label{2.49a}
\Big[\sqrt{-\Delta}-\nu_{N}-N\kappa^{\rm HF}_N\int_{\R^{3}}\frac{\rho_{\gamma }(y)}{|x-y|}dy\Big]\, W(x)\leq0\ \ \text{in}\, \ \R^3.
\end{equation}
Since $\rho_{\gamma }\in L^r(\R^3, \R)$ holds for any $r\in[1, 3]$, one can calculate that
\begin{equation*}
\begin{split}
&\lim\limits_{|x|\to\infty}\int_{\R^{3}}\frac{\rho_{\gamma }(y)}{|x-y|}dy=0.
\end{split}
\end{equation*}
This yields from \eqref{2.49a} that there exists a sufficiently large constant $R>0$ such that
\begin{equation}\label{2.41b}
\Big(\sqrt{-\Delta}-\frac{\nu^{*}_{N}}{2}\Big)\, W(x)\leq0\ \,\ \text{in}\, \ B^c_R.
\end{equation}
Following the comparison argument of \cite[Lemma C.2]{green}, we thus deduce from \eqref{2.39} and \eqref{2.41b} that there exists a constant $C>0$ such that
\begin{equation}\label{2.57b}
\sum_{k=1}^N|w_{k}|=W(x)\leq C\big(1+|x|\big)^{-4}\ \,\ \text{in}\ \, \R^3,
\end{equation}
which further implies that the decaying estimate \eqref{decay1} holds true. The proof of Theorem \ref{th2.1} is therefore complete.
\end{proof}
	
Following the fact that  $\kappa_N^{\rm HF}\to 0$ as $N\to\infty$, we finally employ  Lemma \ref{nomin} and \cite[Theorem 4]{LenLew-10} to complete the proof of Theorem \ref{cor1}.
	
\begin{proof}[Proof of Theorem \ref{cor1}]
For any $m>0$ and $2\leq N\in\mathbb{N}$, it can be verified from  the definition of $\kappa^{\rm HF}_N$  that   $I^{\rm HF}_{m,\kappa}(N)=-\infty$ if and only if $\kappa>\kappa^{\rm HF}_N$. Recall that $\infty=\kappa_{1}^{\rm HF}>\kappa_{2}^{\rm HF}>\kappa_{3}^{\rm HF}>\cdots>\kappa_N^{\rm HF}>0$ and  $\kappa_N^{\rm HF}\to 0$ as $N\to\infty$.
We thus deduce that  $N^{\text{HF}}_\kappa=1$ if $\kappa>\kappa^{\rm HF}_2$,  $N^{\text{HF}}_\kappa=N$ if $\kappa=\kappa^{\rm HF}_N$ for some $2\leq N\in\mathbb{N}$,   and  $N^{\text{HF}}_\kappa=N$ if $\kappa\in (\kappa^{\rm HF}_{N+1}, \kappa^{\rm HF}_N)$ for some $2\leq N\in\mathbb{N}$.
	
If  $\kappa>\kappa^{\rm HF}_2$, then $N^{\rm HF}_\kappa =1$. It thus yields from  \cite[Theorem 4]{LenLew-10} that the Hartree--Fock variational problem $I^{\rm HF}_{m,\kappa}(N^{\rm HF}_\kappa)$ in  \eqref{problem} has no minimizer.
	
If $\kappa=\kappa^{\rm HF}_N$ holds for some $2\leq N\in\mathbb{N}$, then it follows from  Lemma \ref{nomin} that $I^{\rm HF}_{m,\kappa}(N)$  has no miminizer. This implies that  $I^{\rm HF}_{m,\kappa}(N^{\rm HF}_\kappa )$ has no miminizer, due to  the fact that $N^{\rm HF}_\kappa =N$.
	
If $\kappa\in (\kappa_{N+1}^{\rm HF}, \kappa^{\rm HF}_N)$ holds for some $2\leq N\in\mathbb{N}$, then using the  operator inequality
\begin{equation*}\label{3.8c}
\sqrt{-\Delta+m^2}-m\leq \frac{-\Delta}{2m}, \ \ \forall\ m>0,
\end{equation*} 
we can check that $I^{\rm HF}_{m,\kappa}(N)\in(-\infty, 0)$ and  the energy $\mathcal{E}^{\rm HF}_{m,\kappa}$ is coercive on $\cP_N$. The same argument of \cite[Theorem 4]{LenLew-10} then gives that $I^{\rm HF}_{m,\kappa}(N)$ has a miminizer.  This implies that  $I^{\rm HF}_{m,\kappa}(N^{\rm HF}_\kappa )$ has a miminizer, due to the fact that  $N^{\rm HF}_\kappa =N$, which therefore completes the proof of Theorem \ref{cor1}.
\end{proof}


\section{Limiting behavior of Hartree--Fock minimizers}\label{sec:pf-lim-min}

The aim of this section is to prove Theorem \ref{th2} on the limiting behavior of HF minimizers for $I^{\rm HF}_{m,\kappa}(N)$ defined in \eqref{problem} as $\kappa\nearrow \kappa^{\rm HF}_N$, where $\kappa^{\rm HF}_N>0$ is given in Theorem \ref{th2.1}.

\subsection{Limiting optimization problem} 

We first investigate the limiting   problem of the optimal constant $d_N^*$ defined in \eqref{dn}, which is crucial for the proof of Theorem \ref{th2}.


\begin{lem}[Limiting optimization problem]\label{lem3.1a}
For any given $2\leq N\in\mathbb{N}$,  the optimal constant $d_N^*$ of \eqref{dn} satisfies $0<d_{N}^{*}< \infty$ and admits an optimizer. Moreover, if  $\{\gamma_n\}$ is a minimizing sequence of $d_N^*$, then up to a subsequence,
\begin{align}\label{eq:gamman-CV-gamma-kinetic-strong}
(1-\Delta)^{1/4}\gamma_n (1-\Delta)^{1/4} \to (1-\Delta)^{1/4}\gamma(1-\Delta)^{1/4}
\end{align}
strongly  in $\fS^1$ as $n\to\infty$, where $\gamma$ is an optimizer  of $d_N^*$, and the density $\rho_{\gamma_n}(x)$ has the uniform upper bound
\begin{equation}\label{a.13}
\rho_{\gamma_n}(x) \leq C\big(1+|x|\big)^{-8}\ \ \,\text{in}\, \ \R^3\ \, as \,\ n\to\infty.
\end{equation}
\end{lem}

\begin{proof} We first prove that $d_N^* \in (0,\infty)$. Suppose $\gamma$ is an optimizer  of $\kappa^{\rm HF}_N$ defined in \eqref{eq:GN-HF},  then we derive from \eqref{eq:HK} and   \eqref{decay} that $\Tr\left( (-\Delta)^{-1/2} \gamma \right) \in (0,\infty)$ is well-defined. Moreover, it yields from the Cauchy--Schwarz inequality  that
\begin{equation*}\label{eq:Tr-gamma-p-12-CS}
\Tr \big( (-\Delta)^{-1/2}\gamma\big) \Tr ( (-\Delta)^{1/2}\gamma) \ge ( \Tr \gamma)^2,
\end{equation*}
which implies that $d_N^* \in [N^2,\infty)$, and we are done.
	
We next carry out the rest proof by the following three steps.
	
\medskip
\noindent	
{\bf Step 1.}  We prove that if $\{\gamma_n\}$ is a minimizing sequence of $d_N^*$ defined in \eqref{dn}, then up to a subsequence, $\gamma_n$ satisfies
\begin{align}\label{eq:gamman-CV-gamma-kinetic-weak}
	(1-\Delta)^{1/4}\gamma_n (1-\Delta)^{1/4} \rightharpoonup (1-\Delta)^{1/4}\gamma(1-\Delta)^{1/4}
\end{align}
weakly  in $\fS^1$ as $n\to\infty$, where the operator $\gamma=\sum_{j=1}^N |u_j\rangle \langle u_j|\neq0$, and $u_j$ satisfies $\|u_j\|_{L^2}^2\le 1$ and
\begin{equation}\label{mean-lim}
	\begin{split}
		&\sqrt{-\Delta}u_{j}-\kappa^{\rm HF}_N(|x|^{-1}\ast\rho_{\gamma})u_{j}+\kappa^{\rm HF}_N\inte\frac{\gamma(x,y)u_{j}(y)}{|x-y|}dy\\
		&=\nu_ju_{j}\ \ \mbox{in}\ \, \R^3, \ \ \ j=1, \cdots,N,  
\end{split}\end{equation}
together with $\nu_{1} \leq\cdots\leq\nu_{N} \le0$.

To prove above results, we first claim that the vanishing case of $\{\gamma_n\}$ does not occur, namely,
\begin{align}\label{eq:gamma-n-no-vn}
	\liminf_{n\to \infty} \sup_{z\in \R^3} \int_{B_1(z)} \rho_{\gamma_n} (x) d x > 0.
\end{align}
Since   the constraint of \eqref{dn} gives that $\gamma_n \in \cP_N$ and $\Tr (\sqrt{-\Delta} \gamma_n) = 1$, it implies that $\sqrt{\rho_{\gamma_n}}$ is bounded uniformly in $H^{1/2}(\R^3)$ for all $n>0$. Thus, if \eqref{eq:gamma-n-no-vn} were false, then it follows from \cite[Lemma 7.2]{LenLew-10} that $\rho_{\gamma_n}$ would converge strongly to $0$ in $L^p(\R^3)$ for all $1< p < 3/2$. This gives from \eqref{inter-bdd1} that 
$$0 \le \Ex(\gamma_n) \le D(\rho_{\gamma_n},\rho_{\gamma_n}) \to 0 \ \ \mbox{as} \ \ n\to\infty.$$ 
Since ${\gamma_n}$ is an optimizer of $\kappa_N^{\rm HF}$, this yields that
$$
\kappa_N^{\rm HF}= \frac{2 \Tr \big(\sqrt{-\Delta}\gamma_n\big)}{D(\rho_{\gamma_n},\rho_{\gamma_n}) - \Ex(\gamma_n) } \to\infty\ \ \ \text{as}\ \ n\to\infty,$$
a contradiction. This hence proves the claim (\ref{eq:gamma-n-no-vn}).

By Theorem \ref{th2.1}, we can write $\gamma_{n}=\sum_{j=1}^{N}|u_{j}^{n}\rangle\langle u_{j}^{n}|$, where the orthonormal functions $u_1^{n}, \cdots, u_{N}^{n}$ solve the following system:
\begin{align}\label{mean-uj}
	\sqrt{-\Delta}u_{j}^{n}-\kappa^{\rm HF}_N\inte\frac{\rho_{\gamma_{n}}(y)}{|x-y|}dy\, u_{j}^{n}+\kappa^{\rm HF}_N \inte\frac{\gamma_{n}(x,y)u_{j}^{n}(y)}{|x-y|}dy =\nu_{j}^{n}u_{j}^{n}\ \ \mbox{in}\ \ \R^3,
\end{align}
where $ j=1,\cdots, N$ and $\nu_{1}^{n}\leq\cdots\leq\nu_{N}^{n}<0$. Applying the non-vanishing property \eqref{eq:gamma-n-no-vn}, up to a translation,  there exist a subsequence of $\{\gamma_n\}$, still denoted by $\{\gamma_n\}$, and an operator $0\neq\gamma=\sum_{j=1}^N |u_j\rangle \langle u_j|$ such that (\ref{eq:gamman-CV-gamma-kinetic-weak}) holds true, where  the function $u_j $ satisfies $\|u_j\|_{L^2}^2\le 1$ and
\begin{eqnarray}\label{a.11}
	u_{j}^{n}\rightharpoonup u_{j}\ \ \mathrm{weakly\ in}\, \ H^{1/2}(\R^3)\ \ \text{as}\ \ n\to\infty,\ \   j=1,\cdots,N.
\end{eqnarray}
Multiplying \eqref{mean-uj} by $\overline{u_{j}^{n}}$, and  integrating  over $\R^3$, we then obtain that for any $n\in \mathbb{N},$
\begin{align*}
	\sum_{j=1}^{N}\nu_{j}^{n}&=\Tr(\smash{\sqrt{-\Delta}\gamma_{n}})-\kappa^{\rm HF}_ND(\rho_{\gamma_{n}},\rho_{\gamma_{n}})+\kappa^{\rm HF}_N\Ex(\gamma_{n})\\
	&=-\Tr(\smash{\sqrt{-\Delta}\gamma_{n}})=-1,
\end{align*}
where we have used the fact that $\|u_{j}^{n}\|_{L^2}^2=1$ holds for any $j=1, \cdots, N$. This hence implies that up to a subsequence,
\begin{align}\label{mean-ev-limt}
	-\infty<\lim_{n\rightarrow\infty}\nu_{j}^{n}:=\nu_j\leq 0,\ \, j=1,\cdots,N.
\end{align}
We therefore deduce from \eqref{mean-uj}--\eqref{mean-ev-limt} that the operator $\gamma=\sum_{j=1}^N |u_j\rangle \langle u_j|$  satisfies (\ref{mean-lim}), and Step 1 is proved.

\medskip
\noindent	
{\bf Step 2.} We prove that the operator $\gamma$ of Step 1 is indeed an optimizer of $d_N^*$ defined in \eqref{dn}, and \eqref{eq:gamman-CV-gamma-kinetic-strong} holds true.

To establish Step 2, we first claim that the operator $\gamma$ of Step 1 satisfies the following Pohozaev-type identity:
\begin{align}\label{Pohozaev}
	\Tr(\smash{\sqrt{-\Delta}\gamma})&-\frac{5\kappa^{\rm HF}_N}{4}\big[D(\rho_{\gamma},\rho_{\gamma})-\Ex(\gamma)\big]=\frac{3}{2}\sum_{j=1}^{N}\nu_j.
\end{align}
Actually, since $u_{j}\in H^{1/2}(\R^{3})$ satisfies \eqref{mean-lim}, the same argument of \eqref{hs} gives that $u_{j}\in H^{s}(\R^{3})$ holds for any $s\geq 1$, where $j=1,\cdots,N$.  By Sobolev's embedding theorem, this then implies that $u_j\in C^\infty(\R^3)$. Multiply \eqref{mean-lim} by $x\cdot\nabla \overline{u}_{j}$, integrate it over $\R^{3}$, and then take its real part for  any $j=1, \cdots, N$. We thus have
\begin{align}\label{kin-rel}
	\Re\langle x\cdot\nabla u_{j},\sqrt{-\Delta}u_{j}\rangle 
	&=-\langle u_{j},\sqrt{-\Delta}u_{j}\rangle,\,\ j=1, \cdots, N,\\
	\qquad\ \Re\langle x\cdot\nabla u_{j},u_{j}\rangle &=-\frac{3}{2}\|u_{j}\|_{2}^{2},\,\ j=1, \cdots, N,
\end{align}
and
\begin{align}\label{5/4-dir}
	&\Re\sum_{j=1}^{N}\int_{\R^{3}}\int_{\R^{3}}\frac{\rho_{\gamma}(y)}{|x-y|}u_{j}(x)\big(x\cdot\nabla\bar {u}_{j}(x)\big)dxdy\nonumber\\ &\quad=\frac{1}{2}\int_{\R^{3}}\big(|x|^{-1}*\rho_{\gamma}\big)\big(x\cdot\nabla\rho_{\gamma}\big)dx=-\frac{5}{4}D(\rho_{\gamma},\rho_{\gamma}).
\end{align}
Similar to \eqref{5/4-dir}, one can also calculate that
\begin{align}\label{5/4-ex} &\Re\sum_{j=1}^{N}\int_{\R^{3}}\int_{\R^{3}}\frac{\gamma(x,y)u_{j}(y)}{|x-y|}\big(x\cdot\nabla\bar{u}_{j}(x)\big)dxdy\nonumber\\
	& =\frac{1}{2}\sum_{j,k=1}^N\int_{\R^{3}}\big[|x|^{-1}\ast(u_{j}\bar u_{k})\big]\big(x\cdot\nabla\bar{u}_{j}u_{k}\big)dx\\
	& =-\frac{5}{4}\sum_{j,k=1}^{N}\int_{\R^{3}}(|x|^{-1}\ast u_{j}\bar u_{k})\bar u_{j}u_{k}dx=-\frac{5}{4}\Ex(\gamma).\nonumber
\end{align}
We hence conclude from \eqref{kin-rel}--\eqref{5/4-ex} that the claim \eqref{Pohozaev} holds true.

Applying \eqref{mean-lim} and \eqref{Pohozaev}, we now get that
$$\Tr(\smash{\sqrt{-\Delta}\gamma})=\frac{\kappa^{\rm HF}_N}{2}\big[D(\rho_{\gamma},\rho_{\gamma})-\Ex(\gamma)\big]\neq 0.$$
Note from Lemmas \ref{than1} and \ref{nomin} that $\kappa_N^{\rm HF}=\kappa_N^*$, where  $\kappa_N^*$ defined in \eqref{an}. We then have
\begin{align}\label{5.15cb}
	\frac{2\|\gamma\|\Tr(\smash{\sqrt{-\Delta}\gamma})}{D(\rho_{\gamma},\rho_{\gamma})-\Ex(\gamma)}
	\geq\kappa_N^{*}=\frac{2\Tr(\smash{\sqrt{-\Delta}\gamma})}{D(\rho_{\gamma},\rho_{\gamma})-\Ex(\gamma)}.
\end{align}
Because $\|\gamma\|\leq\liminf_{n\rightarrow\infty}\|\gamma_{n}\|=1$, we deduce from \eqref{5.15cb}  that  $\|\gamma\|=1$, and $\gamma\in \cR_N$  is an optimizer of $\kappa_N^{*}$. Together with Lemma \ref{than1}, we thus conclude that $\gamma\in \cP_N$  is an optimizer of $\kappa_N^{\rm HF}$, and 
\begin{equation}\label{5.16cb}
	u_{j}^{n}\to  u_{j}\ \ \ \text{strongly in}\  L^2(\R^3)\ \ \text{as}\  \ n\rightarrow\infty, \ \ \ j=1,\cdots,N.
\end{equation}
Since $\sqrt{\rho_{\gamma_n}}$ is bounded uniformly in $H^{\frac{1}{2}}(\R^3)$, which further implies that $\rho_{\gamma_n}\to \rho_\gamma$ strongly in $L^p(\R^3)$ as $n\to\infty$ for all $1\le p<3/2$. Therefore,  we derive from  \eqref{inter-bdd1} that
$$0\le X(\gamma_n-\gamma)\le D(\rho_{\gamma_n-\gamma},\rho_{\gamma_n-\gamma})\to 0\ \  \text{as}\ \ n\to\infty,$$
and hence
\begin{align}\label{DEx-con}
	D(\rho_{\gamma_{n}},\rho_{\gamma_{n}}) \to D(\rho_{\gamma},\rho_{\gamma}),\ \ \Ex(\gamma_n)\to \Ex(\gamma)\ \ \mbox{as}\ \ n\to\infty.
\end{align}
Since $\gamma_{n}$ and $\gamma$ are optimizers of  $\kappa^{\rm HF}_N$,  by \eqref{DEx-con} and Fatou's lemma,  we have
\begin{align}
	\Tr(\smash{\sqrt{-\Delta}\gamma})&\leq \liminf_{n\rightarrow\infty}\Tr(\smash{\sqrt{-\Delta}\gamma_{n}}) \nonumber\\
	&=\frac{\kappa^{\rm HF}_N}{2}\liminf_{n\rightarrow\infty}\big(D(\rho_{\gamma_{n}},\rho_{\gamma_{n}})-\Ex(\gamma_{n})\big)\\
	&=\frac{\kappa^{\rm HF}_N}{2}\big(D(\rho_{\gamma},\rho_{\gamma})-\Ex(\gamma)\big)=\Tr(\smash{\sqrt{-\Delta}\gamma}).\nonumber
\end{align}
This proves that 
\begin{align}\label{5.19}
	\Tr(\smash{\sqrt{-\Delta}\gamma})=\liminf_{n\rightarrow\infty}\Tr(\smash{\sqrt{-\Delta}\gamma_{n}})=1. 
\end{align}
Combining \eqref{a.11}, \eqref{5.16cb} and \eqref{5.19},  we thus  derive that   \eqref{eq:gamman-CV-gamma-kinetic-strong} holds true, $i.e.$, $u_{j}^{n}\rightarrow u_{j}$ strongly in $H^{1/2}(\R^{n})$ as $n\rightarrow\infty$ for each $j=1,\cdots,N$. 
Moreover, since $\{\gamma_n\}$ is a minimizing sequence of $d_N^*$ and $\gamma$  is an optimizer of $\kappa_N^{\rm HF}$, using again  Fatou's lemma, we have
$$
\Tr\big( (-\Delta)^{-1/2}\gamma\big)\le \liminf_{n\to \infty}\Tr\big( (-\Delta)^{-1/2}\gamma_n\big) = d_N^*\leq\Tr\big( (-\Delta)^{-1/2}\gamma\big).
$$
We thus conclude that $\gamma$ is also an optimizer of $d_N^*$ defined in \eqref{dn}.  This proves Step 2.

\medskip
\noindent	
{\bf Step 3.} It remains to prove the uniform upper bound (\ref{a.13}) of the density $\rho_{\gamma_n}(x)$ as $n\to \infty$. In the
following, all constants $C>0$ are independent of $n>0$.

Towards the above aim, we first claim that for $j=1, \cdots, N$,
\begin{align}\label{uj-H1bdd}
	\text{ the sequence }	\{u_{j}^{n}\}_{n} \text{ is bounded uniformly in }H^{1}(\R^{3})\,  \text{ as }\, n\to \infty.
\end{align}
Indeed, we follow from \eqref{mean-uj} that for $j=1, \cdots, N$, 
\begin{equation}\label{eq:ujn-H1}
\begin{split}	
	&\|\sqrt{-\Delta}u_{j}^{n}\|_{L^2}^2 \\
	=&\Big\|\nu_{j}^{n}u_{j}+\kappa^{\rm HF}_N(|x|^{-1}\ast\rho_{\gamma_{n}})u_{j}^{n}-\kappa^{\rm HF}_N \inte\frac{\gamma_{n}(x,y)u_{j}^{n}(y)}{|x-y|}dy\Big\|_{L^2}^{2}.
\end{split}\end{equation}
Since $\sqrt{\rho_{\gamma_n}}$ is bounded uniformly in $H^{1/2}(\R^3)$, we have
\begin{align}\label{kin-bdd}
	0&\leq \sup_{x\in\R^{3}}\int_{\R^{3}}\frac{\rho_{\gamma_{n}}(y)}{|x-y|}dy\leq \frac{\pi}{2}\big\langle \sqrt{\rho_{\gamma_n}}, \sqrt{-\Delta} \sqrt{\rho_{\gamma_n}} \big\rangle \leq C<\infty,
\end{align}
due to the Hardy--Kato inequality \eqref{eq:HK}.
Moreover, since $\Tr(\sqrt{-\Delta}\gamma_n)=1$, by the Cauchy-Schwarz inequality,  we have
\begin{align}\label{ex-bdd}
	0&\leq 
	\left| \int_{\R^{3}}\frac{\gamma_{n}(x,y)u_{j}^{n}(y)}{|x-y|}dy \right| \nn\nonumber\\
	&\leq \sum_{k=1}^{N}|u_{k}^{n}(x)|\left(\int_{\R^{3}}\frac{|u_{k}(y)|^{2}}{|x-y|}dy\right)^{1/2}\left(\int_{\R^{3}}\frac{|u_{j}(y)|^{2}}{|x-y|}dy\right)^{1/2}\\
	&\leq \Tr(\smash{\sqrt{-\Delta}\gamma_{n}})\sum_{k=1}^{N}|u_{k}^{n}(x)|\leq \sum_{k=1}^{N}|u_{k}^{n}(x)|, \ \ j=1, \cdots, N.\nonumber
\end{align}
Inserting \eqref{kin-bdd} and \eqref{ex-bdd} into \eqref{eq:ujn-H1}, we then conclude that the claim \eqref{uj-H1bdd} holds true.

We next claim that for $j=1, \cdots, N$,
\begin{align}\label{f-str-con}
	f_{j}^{n}\rightarrow f_{j}\,\  \text{strongly in }\,L^{p}(\R^{3})\, \text{ as }\, n\rightarrow\infty,\ \ \forall \,2\leq p<6,
\end{align}
where we denote, for each $j=1,...,N$,
\begin{equation}\label{f5:21}
	\begin{split}
		\frac{f_{j}^{n}(x)}{\kappa^{\rm HF}_N}&:=\int_{\R^{3}}\frac{\rho_{\gamma_{n}}(y)}{|x-y|}dy|u_{j}^{n}(x)|+
		\int_{\R^{3}}\frac{|\gamma_{n}(x,y)|}{|x-y|}|u_{j}^{n}(y)|dy,\\
		\frac{f_{j}(x)}{\kappa^{\rm HF}_N}&:=\int_{\R^{3}}\frac{\rho_{\gamma}(y)}{|x-y|}dy|u_{j}^{n}(x)|+\int_{\R^{3}}\frac{|\gamma(x,y)|}{|x-y|}|u_{j}^{n}(y)|dy.
	\end{split}
\end{equation}
Indeed, since $u_{j}^{n}\to u_j$ strongly in $H^{1/2}(\R^3)$ as $n\to \infty$, we have, for each $j=1, \cdots, N,$
\begin{align*} 
	&\int_{\R^{3}}\Big(\int_{\R^{3}}\frac{\rho_{{\gamma_{n}}}(y)}{|x-y|}dy\, |u_{j}^{n}(x)|-\int_{\R^{3}}\frac{\rho_{\gamma}(y)}{|x-y|}dy\, |u_{j}(x)|\Big)^{2}dx\nonumber\\
	&\leq2\int_{\R^{3}}\Big|\int_{\R^{3}}\frac{\rho_{{\gamma_{n}}}(y)}{|x-y|}dy\Big|^{2}\, \big|u_{j}^{n}(x)-u_{j}(x)\big|^{2}dx\nonumber\\
	&\quad+2\int_{\R^{3}}\Big|\int_{\R^{3}}\frac{\rho_{{\gamma_{n}}}(y)-\rho_{\gamma}(y)}{|x-y|}dy\Big|^{2}\, \big|u_{j}(x)\big|^{2}dx \to 0\,  \text{ as }\, n\to \infty,
\end{align*}
where the Hardy--Littlewood--Sobolev inequality is employed. Similarly, we have for $j=1, \cdots, N,$
\begin{align*} 
	&\int_{\R^{3}}\Big|\int_{\R^{3}}\frac{|\gamma_{n}(x,y)|\, 	|u_{j}^{n}(y)|-|\gamma(x,y)|\, |u_{j}(y)|}{|x-y|}dy\Big|^{2}dx\nonumber\\
	&\leq2\int_{\R^{3}}\Big(\int_{\R^{3}}\frac{\big|\gamma_{n}(x,y)\big|\, \big|u_{j}^{n}(y)-u_{j}(y)\big|}{|x-y|}dy\Big)^{2}dx\nonumber\\
	&\quad+2\int_{\R^{3}}\Big(\int_{\R^{3}}\frac{\big|\gamma_{n}(x,y)-\gamma(x,y)\big|\, |u_{j}(y)|}{|x-y|}dy\Big)^{2}dx \to 0\,  \text{ as }\, n\to \infty.
\end{align*}
These yield that $f_{j}^{n}\to   f_{j}$ strongly in $L^{2}(\R^3)$ as $n\to\infty$, where $f_{j}^{n}$ and $f_j$ are as in (\ref{f5:21}) for $j=1, \cdots, N$. Moreover, since
$$|f_{j}^{n}(x)|\leq 2C\kappa^{\rm HF}_N\sum_{k=1}^N \big|u_{k}^n(x)\big|\ \ \mbox{in}\,\ \R^3,$$
and $\{u_k^n\}_n$ is bounded uniformly in $H^1(\R^3)$, it gives that $f_j^n$ is bounded uniformly in $L^6(\R^3)$ for $j=1, \cdots, N$. By the interpolation inequality, we thus obtain the claim \eqref{f-str-con}.

We now prove that for each $j=1, \cdots, N$,
\begin{align}\label{Linf-ubdd}
	\|u_{j}^{n}\|_{L^\infty}\text{ is bounded uniformly for sufficiently large }n>0,
\end{align}
and
\begin{align}\label{asymp-udecay}
	\lim_{|x|\rightarrow\infty}|u_{j}^{n}(x)|=0\quad\text{ uniformly for sufficiently large }n>0.
\end{align}
Actually, the argument of \cite[Lemma 2.4]{me3} yields from \cite[Section 7.11 (9)]{analysis} that the Green function $G_{j}^n$ of the operator  $\sqrt{-\Delta}-\nu_j^n$ in $\R^3$ has the following uniform bound for sufficiently large $n>0$:
\begin{align}\label{g1}
	0<G_{j}^{n}(x)<C |x|^{-2} \ \ \text{uniformly in}\ \, \R^3
\end{align}
due to the fact that
\begin{align}\label{3.49}
	\lim\limits_{n\to\infty}\nu_{j}^{n}=\nu_j<0,\ \ j=1,\cdots,	N,
\end{align}
where $C>0$ is independent of $j, n$ and $x$.
Applying \eqref{f-str-con} and \eqref{g1}, we then deduce from \eqref{mean-uj} that, for $j=1, \cdots, N,$
\begin{align*}
	\|u_{j}^{n}\|_{L^\infty}&= \| G_j^n * f_j^n\|_{L^\infty}\\
    &\leq C \sup\limits_{x\in\R^3}\Big[\int_{|x-y|\leq1}|x-y|^{-2}f_{j}^{n}(y) dy+\int_{|x-y|>1}|x-y|^{-2}f_{j}^{n}(y) dy\Big]\\
	&\leq C\big\|f_{j}^{n}\big\|_{L^5}\Big(\int_{|y|\leq1}|y|^{-\frac{5}{2}}dy\Big)^{\frac{4}{5}}
	+C_g\big\|f_{j}^{n}\big\|_{L^{5/2}}\Big(\int_{|y|>1}|y|^{-\frac{10}{3}}dy\Big)^{\frac{3}{5}}\\
	&\leq C\big\|f_{j}^{n}\big\|_{L^5}+C\big\|f_{j}^{n}\big\|_{L^{5/2}} \le C
	\ \ \mathrm{uniformly\ for \ sufficiently\ large}\ n>0,
\end{align*}
and 
\begin{align*}
	&\lim\limits_{|x|\to\infty} |u_{j}^{n}(x)|\\ &\leq\lim\limits_{R\to\infty}\lim\limits_{|x|\to\infty}\Big( \int_{|y|\leq R} G_{j}^{n}(x-y)f_{j}^{n}(y) dy+\int_{|y|\geq R} G_{j}^{n}(x-y)f_{j}^{n}(y)dy\Big)\\
	&=\lim\limits_{R\to\infty}\lim\limits_{|x|\to\infty} \int_{|y|\geq R} G_{j}^{n}(x-y)f_{j}^{n}(y)dy\\[1.5mm]
	&\leq
	\lim\limits_{R\to\infty}\lim\limits_{|x|\to\infty}\int_{\{y:\, |y|\geq R, \, |x-y|\leq 1\}}
	G_{j}^{n}(x-y)f_{j}^{n}(y)dy\\
	&\quad\quad+\lim\limits_{R\to\infty}\lim\limits_{|x|\to\infty}\int_{\{y:\, |y|\geq R, \, |x-y|>1\}}
	G_{j}^{n}(x-y)f_{j}^{n}(y)dy\\
	&\leq  C \lim\limits_{R\to\infty}\big\|f_{j}^{n}\big\|_{L^5(B_R^c)}+ C \lim\limits_{R\to\infty}\big\|f_{j}^{n}\big\|_{L^{\frac{5}{2}}(B_R^c)}=0
\end{align*}
uniformly for  sufficiently large $n>0$. This therefore shows that \eqref{Linf-ubdd} and \eqref{asymp-udecay} hold true.

Applying \eqref{mean-uj}, \eqref{Linf-ubdd},  \eqref{asymp-udecay} and \eqref{3.49},  similar to \eqref{2.57b}, one can derive that \eqref{a.13} holds  uniformly in $n>0$.  The proof of Lemma \ref{lem3.1a} is therefore complete.
\end{proof}

\subsection{Proof of Theorem \ref{th2}}\label{sec:pf-thmth2}
The purpose of this subsection is to complete the proof of Theorem \ref{th2}. For this reason, we first recall that $I^{\rm HF}_{m,\kappa_{N}^{\rm HF}}(N)+mN=0$. The following lemma is concerned with the convergence rate of $I^{\rm HF}_{m,\kappa_{n}}(N)+mN$ and the blow-up rate of the corresponding minimizers, where $\kappa_{n}\nearrow \kappa^{\rm HF}_N$ as $n\to\infty$.

\begin{lem}[Energy estimate]\label{lem3.1}   For $m>0$ and $2\leq N\in\mathbb{N}$, suppose that $\gamma_{\kappa_{n}}=\sum_{j=1}^{N}\big|u_{j}^{\kappa_{n}}\big\rangle \big\langle u_{j}^{\kappa_{n}}\big|$ is a minimizer of $I^{\rm HF}_{m,\kappa_{n}}(N)$ defined by \eqref{problem}, where $\kappa_{n}\nearrow \kappa^{\rm HF}_N$ as $n\to\infty$, and $u_{1}^{\kappa_{n}}, \cdots, u_{N}^{\kappa_{n}}$ are  eigenfunctions of the operator
\begin{equation*} H_{\gamma_{\kappa_{n}}}:=\sqrt{-\Delta+m^{2}}-m-\kappa_{n}\int_{\R^{3}}\frac{\rho_{\gamma_{\kappa_{n}}}(y)}{|x-y|}dy +\kappa_{n}\, \frac{\gamma_{\kappa_{n}}(x,y)}{|x-y|}\, \text{ on }\, L^2(\R^3).
\end{equation*}
Then, up to a subsequence if necessary, there exist a sequence $\{z_n\}\subset\R^3$ and an orthonormal system $\{w_1, \cdots, w_N\}$ such that for $j=1,\cdots,N,$
\begin{eqnarray}\label{3.0}	 w_{j}^{\kappa_{n}}(x):=\epsilon_{\kappa_{n}}^{3/2}u_{j}^{\kappa_{n}}\big(\epsilon_{\kappa_{n}}x+\epsilon_{\kappa_{n}}z_n\big)\to w_{j}
\end{eqnarray}
strongly in $H^{1/2}(\R^{3})$ as $n\rightarrow\infty$, 
where the constant $\epsilon_{\kappa_{n}}>0$ satisfies
\begin{align}\label{4.6}
\epsilon_{\kappa_{n}}:=\Big(\Tr\big(\sqrt{-\Delta}\, \gamma_{\kappa_{n}}\big)\Big)^{-1}\to0\ \ \text{as }\ n\to\infty,
\end{align}
and $\gamma^{*}:=\sum_{j=1}^{N}|w_{j}\rangle \langle w_{j}|$ is an optimizer of $d_N^*$ defined in \eqref{dn}. Moreover,  we have the asymptotic formulas
\begin{equation}\label{52.MN}
\begin{split}
\epsilon_{\kappa_{n}}&\sim \Big[\frac{2( \kappa^{\rm HF}_N-\kappa_{n})}{m^2\kappa^{\rm HF}_N d_N^*}\Big]^{1/2} \ \ \text{as}\ \ n\to\infty,\\
I^{\rm HF}_{m,\kappa_{n}}(N)+mN&\sim m \Big[\frac{2d_N^*(\kappa^{\rm HF}_N-\kappa_{n})}{\kappa^{\rm HF}_N}\Big]^{1/2}\ \ \text{as}\ \ n\to\infty.
\end{split}
\end{equation}
\end{lem}

\begin{proof} 
 Lemma \ref{lem3.1} is proved by the following two steps.
 
\medskip
\noindent	
{\bf Step 1.} 
We prove that both (\ref{3.0}) and \eqref{4.6} hold true. For this purpose, since $\kappa_n<\kappa_N^{\rm HF}$, we first observe that
$I^{\rm HF}_{m,\kappa^{\rm HF}_N}(N)\le I^{\rm HF}_{m,\kappa_{n}}(N)$. On the other hand, we have for any $\gamma\in\cP_{N}$ given by \eqref{set},
\begin{equation} \label{52.M}
	\begin{split}
		\cE^{\rm HF}_{m,\kappa^{\rm HF}_N}(\gamma)
		&=\cE^{\rm HF}_{m,\kappa_{n}}(\gamma)-\frac{\kappa^{\rm HF}_N-\kappa_{n}}{2} \big(D(\rho_\gamma,\rho_\gamma)-\Ex(\gamma)\big) \\
		&\geq I^{\rm HF}_{m,\kappa_{n}}(N)-\frac{\kappa^{\rm HF}_N-\kappa_{n}}{2} \big(D(\rho_\gamma,\rho_\gamma)-\Ex(\gamma)\big).
\end{split}\end{equation}
Optimizing (\ref{52.M}) over $\gamma$,  we then obtain that
\begin{align}\label{3.1}
	I^{\rm HF}_{m,\kappa^{\rm HF}_N}(N) = \lim\limits_{n\to\infty}I^{\rm HF}_{m,\kappa_{n}}(N)=\lim\limits_{n\to\infty}\cE^{\rm HF}_{m,\kappa_{n}}(\gamma_{\kappa_{n}})\in(-\infty, 0).
\end{align}

On the contrary, we now suppose that  the limit \eqref{4.6} is false. Then $\Tr\big(\sqrt{-\Delta}\, \gamma_{\kappa_{n}}\big)$ is bounded uniformly in $n$. This gives that
\begin{equation}\label{4.12a}
	\begin{split}
		0&=\lim\limits_{n\to\infty}\frac{\kappa^{\rm HF}_N-\kappa_{n}}{\kappa^{\rm HF}_N}\Tr\big(\sqrt{-\Delta}\, \gamma_{\kappa_{n}}\big)\\
		&\geq\lim\limits_{n\to\infty}\frac{\kappa^{\rm HF}_N-\kappa_{n}}{2}  \big(D(\rho_{\gamma_{\kappa_n}},\rho_{\gamma_{\kappa_n}})-\Ex(\gamma_{\kappa_n})\big) \geq 0,
	\end{split}
\end{equation}
and hence
$$\lim\limits_{n\to\infty}\cE^{\rm HF}_{m,\kappa_{n}}(\gamma_{\kappa_{n}})
=\lim\limits_{n\to\infty}\cE^{\rm HF}_{m,\kappa^{\rm HF}_N}(\gamma_{\kappa_{n}}).$$
We then derive from \eqref{3.1} that
\begin{align*}
	0>I^{\rm HF}_{m,\kappa^{\rm HF}_N}(N)
	=\lim\limits_{n\to\infty}\cE^{\rm HF}_{m,\kappa_{n}}(\gamma_{\kappa_{n}})
	=\lim\limits_{n\to\infty}\cE^{\rm HF}_{m,\kappa^{\rm HF}_N}(\gamma_{\kappa_{n}})\geq I^{\rm HF}_{m,\kappa^{\rm HF}_N}(N),
\end{align*}
which further yields that
\begin{align}\label{4.9}
	0>I^{\rm HF}_{m,\kappa^{\rm HF}_N}(N)=\lim\limits_{n\to\infty}\cE^{\rm HF}_{m,\kappa^{\rm HF}_N}(\gamma_{\kappa_{n}}).
\end{align}
Using the unform boundedness of $\Tr(\sqrt{-\Delta}\gamma_{\kappa_n})$, together with \eqref{4.9}, the  same argument of \cite[Theorem 4]{LenLew-10} then gives that up to translations and a subsequence if necessary,
$$
u_{j}^{\kappa_{n}}\to u_{j}\ \ \text{strongly\ in}\, \ H^{1/2}(\R^3)\ \ \text{as}\ \ n\to\infty,\ \ j=1, \cdots, N,
$$
and $\sum_{j=1}^{N}|u_{j}\rangle\langle u_{j}|$ is a minimizer  of  $I^{\rm HF}_{m,\kappa^{\rm HF}_N}(N)$. This however contradicts with Lemma \ref{nomin}. Thus, \eqref{4.6} holds true.

By the definition of $\epsilon_{\kappa_{n}}>0$ in \eqref{4.6}, we deduce from \eqref{3.1} and \eqref{eq:GN-HF} that
\begin{align}\label{4.7}
	1&=  \lim\limits_{n\to\infty}\epsilon_{\kappa_{n}} \Tr\big(\sqrt{-\Delta}\, \gamma_{\kappa_{n}} )\nonumber\\
	&=  \lim\limits_{n\to\infty}\epsilon_{\kappa_{n}}\Big(-I^{\rm HF}_{m,\kappa_{n}}(N)-mN+\Tr\big(\sqrt{-\Delta+m^{2}}\, \gamma_{\kappa_{n}}\big)\Big) \\
	&=\lim\limits_{n\to\infty}\epsilon_{\kappa_{n}}\frac{\kappa^{\rm HF}_N}{2}\int_{\R^{3}}\int_{\R^{3}}\frac{\rho_{\gamma_{\kappa_{n}}}(x)\rho_{\gamma_{\kappa_{n}}}
		(y)-|\gamma_{\kappa_{n}}(x,y)|^{2}}{|x-y|}dxdy.\nonumber
\end{align}
The similar proof of Lemma \ref{lem3.1a} then yields that the $H^{1/2}$-convergence  \eqref{3.0}  holds true, and 
\begin{equation}\label{4.14}
	\sum_{j=1}^{N}|w_{j}^{\kappa_{n}}(x)|\leq C\big(1+|x|\big)^{-4}\  \,\text{ uniformly in}\ \, \R^3 \  \,\text{for sufficiently large}\ \, n>0,
\end{equation}
where $w_{j}^{\kappa_{n}}$ is as in \eqref{3.0}.  By the definitions of $w_{j}^{\kappa_{n}}$ and $w_{j}$,  we can check from  \eqref{3.0}  and  \eqref{4.7} that  $\gamma^{*}:=\sum_{j=1}^{N}|w_{j}\rangle\langle w_{j}|$ is an optimizer of $\kappa_N^{\rm HF}$, and
\begin{equation*}\label{4.16}
	\Tr\big(\sqrt{-\Delta}\, \gamma^{*}\big)=\lim\limits_{n\to\infty}\epsilon_{\kappa_{n}}\Tr\big(\sqrt{-\Delta}\, \gamma_{\kappa_{n}}\big)=1.
\end{equation*}
This proves Step 1.

\medskip
\noindent	
{\bf Step 2.} We prove that  $\gamma^*$ is an optimizer of $d_N^*$ defined in \eqref{dn} and  the energy estimate (\ref{52.MN}) holds true. Essentially, it yields from \eqref{decay}, \eqref{3.0} and \eqref{4.14} that
\begin{align*} &\lim\limits_{n\to\infty}\int_{\R^{3}}\int_{\R^{3}}|x-y|\Big(\sum_{j=1}^{N}|w_{j}^{\kappa_{n}}(x)|^{2}\sum_{j=1}^{N}|w_{j}^{\kappa_{n}}(y)|^{2} -\Big|\sum_{j=1}^{N}w_{j}^{\kappa_{n}}(x)\overline{w_{j}^{\kappa_{n}}(y)}\ \Big|^{2}\ \Big)dxdy\\ 
	&=\int_{\R^{3}}\int_{\R^{3}}|x-y|\Big(\rho_{\gamma^{*}}(x)\rho_{\gamma^{*}}(y)-|\gamma^{*}(x,y)|^{2}\Big)dxdy<\infty,
\end{align*}
and it follows from \eqref{4.7} that
\begin{align*} \int_{\R^{3}}\int_{\R^{3}}\frac{\rho_{\gamma_{\kappa_{n}}}(x)\rho_{\gamma_{\kappa_{n}}}(y)-|\gamma_{\kappa_{n}}(x,y)|^{2}}{|x-y|}dxdy=\epsilon_{\kappa_{n}}^{-1}\frac{2}{\kappa^{\rm HF}_N}\big(1+o(1)\big) \ \ \text{as}\,\ n\to\infty.
\end{align*}
We then deduce from \eqref{3.0} and \eqref{4.14} that
\begin{align}\label{4.12}
	&I^{\rm HF}_{m,\kappa_{n}}(N)+mN=\cE^{\rm HF}_{m,\kappa_{n}}(\gamma_{\kappa_{n}})+mN\nonumber\\
	&=\Tr\Big[\big(\sqrt{-\Delta+m^{2}}-\sqrt{-\Delta}\, \big)\gamma_{\kappa_{n}}\Big]\nonumber\\
	&\quad\quad+\frac{\kappa^{\rm HF}_N-\kappa_{n}}{2}\int_{\R^{3}}\int_{\R^{3}}\frac{\rho_{\gamma_{\kappa_{n}}}(x)\rho_{\gamma_{\kappa_{n}}}(y)-|\gamma_{\kappa_{n}}(x,y)|^{2}}{|x-y|}dxdy\\
	&\quad\quad+\Tr\big(\sqrt{-\Delta}\, \gamma_{\kappa_{n}}\big)-\frac{\kappa^{\rm HF}_N}{2}\int_{\R^{3}}\int_{\R^{3}}\frac{\rho_{\gamma_{\kappa_{n}}}(x)\rho_{\gamma_{\kappa_{n}}}(y)-|\gamma_{\kappa_{n}}(x,y)|^{2}}{|x-y|}dxdy\nonumber\\
	&\geq\big(1+o(1)\big)\frac{\epsilon_{\kappa_{n}}m^{2}}{2}\Tr\left( (-\Delta)^{-1/2} \gamma^*\right)+\big(1+o(1)\big)\epsilon_{\kappa_{n}}^{-1}\frac{\kappa^{\rm HF}_N-\kappa_{n}}{\kappa^{\rm HF}_N}\nonumber
\end{align}
as $ n\to\infty$. Applying the Cauchy--Schwarz inequality, we thus get from \eqref{4.12} that
\begin{equation}\label{4.17}
\begin{split}
	\lim\limits_{n\to\infty}\frac{I^{\rm HF}_{m,\kappa_{n}}(N)+mN}{\big(\kappa^{\rm HF}_N-\kappa_{n}\big)^{1/2}}
	\geq& m \Big[ \frac{2}{\kappa^{\rm HF}_N} \Tr\big((-\Delta)^{-1/2}\gamma^*\big)\Big]^{1/2} \\
	\ge& m \Big[ \frac{2}{\kappa^{\rm HF}_N} d_N^* \Big]^{1/2},
\end{split}
\end{equation}
where the identities hold if and only if $\gamma^*$ is an optimizer of $d_N^*$, and
\begin{align}\label{4.22c}
	\epsilon_{\kappa_{n}}=\big(1+o(1)\big)m^{-1}\Big[\frac{2( \kappa^{\rm HF}_N-\kappa_{n})}{ \kappa^{\rm HF}_N d_N^*}\Big]^{1/2}\ \ \text{as}\ \ n\to\infty.
\end{align}

We next prove the  upper bound of \eqref{4.17}. Let $\gamma$ be an optimizer of $d_N^*$. One can verify  that for any $t_{n}\in\R^+$, the rescaling operator $\gamma_{t_{n}}(x, y):=t_{n}^3\gamma(t_{n}x, t_{n}y)$ satisfies \begin{align}\label{4.22a}
	D(\rho_{\gamma_{t_{n}} },\rho_{\gamma_{t_{n}} }) -\Ex(\gamma_{t_{n}} )
	= t_{n} \big( D(\rho_{\gamma },\rho_{\gamma}) -\Ex(\gamma) \big) = \frac{2t_{n}}{\kappa^{\rm HF}_N},
\end{align}
where
the second identity  follows from the fact that $\gamma$ is also an optimizer of $\kappa_N^{\rm HF}$ satisfying $\Tr(\sqrt{-\Delta}\gamma)=1$. Moreover,
\begin{align}\label{4.22aaa}
	\Tr\big[\big(\sqrt{-\Delta+m^{2}}-\sqrt{-\Delta}\, \big) \gamma_{t_{n}}\big] \le \Tr\left( \frac{m^2}{2\sqrt{-\Delta}} \gamma_{t_{n}} \right) = \frac{m^2}{2t_n }d_N^*,\ \ \forall \,n\in \mathbb{N}.
\end{align}
We then deduce from
\eqref{4.22a} and \eqref{4.22aaa} that
\begin{align}\label{4.23a}
	I^{\rm HF}_{m,\kappa_{n}}(N)+mN
	\leq&\cE^{\rm HF}_{m,\kappa_{n}}\big(\gamma_{t_{n}}\big) + mN\nonumber\\
	=&\Tr\big[\big(\sqrt{-\Delta+m^{2}}-\sqrt{-\Delta}\, \big) \gamma_{t_{n}}\big]\nn\\
	&+\frac{\kappa^{\rm HF}_N-\kappa_{n}}{2} \big( D(\rho_{\gamma_{t_{n}} },\rho_{\gamma_{t_{n}} }) -\Ex(\gamma_{t_{n}} ) \big) \\
	\leq& \frac{m^{2}}{2t_n} d_N^*+t_{n}\frac{\kappa^{\rm HF}_N-\kappa_{n}}{\kappa^{\rm HF}_N},\ \ \forall\, n\in \mathbb{N}.\nonumber
\end{align}
Choosing
$$
t_{n}=m\Big(\frac{\kappa^{\rm HF}_N d_N^*}{2( \kappa^{\rm HF}_N-\kappa_{n})}\Big)^{1/2}>0,
$$
we thus obtain  from \eqref{4.23a} that
\begin{align}\label{4.19}
	\lim\limits_{n\to\infty}\frac{I^{\rm HF}_{m,\kappa_{n}}(N)+mN}{\big(\kappa^{\rm HF}_N-\kappa_{n}\big)^{1/2}}
	\leq m \Big[ \frac{2}{\kappa^{\rm HF}_N} d_N^* \Big]^{1/2},
\end{align}
which therefore gives the upper bound of \eqref{4.17}.

We now conclude from \eqref{4.17} and \eqref{4.19} that
\begin{align*}
	\lim\limits_{n\to\infty}\frac{I^{\rm HF}_{m,\kappa_{n}}(N)+mN}{\big(\kappa^{\rm HF}_N-\kappa_{n}\big)^{1/2}}
	=m \Big[ \frac{2}{\kappa^{\rm HF}_N} d_N^* \Big]^{1/2},
\end{align*}
which further yields that $\gamma^*$ is an optimizer of $d_N^*$, and  \eqref{4.22c} holds true. This proves the energy estimate (\ref{52.MN}), and the proof of Lemma \ref{lem3.1} is therefore complete.
\end{proof}

\linespread{0}
We are finally ready to complete the proof of Theorem \ref{th2}.

\begin{proof}[Proof of Theorem \ref{th2}.]
Let $w_{j}^{\kappa_{n}}$ and $ w_{j}$ be given by Lemma \ref{lem3.1}. In view of Lemma \ref{lem3.1}, the rest is to prove the $L^\infty$-convergence of the sequence $\{w_{j}^{\kappa_{n}}\}_{n\geq 1}$ as $n\to\infty$, $j=1, \cdots, N$.

Denote $\gamma^{*}_{\kappa_{n}}:=\sum_{j=1}^{N}|w_{j}^{\kappa_{n}}\rangle\langle w_{j}^{\kappa_{n}}|$. Similar to  \eqref{3.49},  it yields from \eqref{equation} and Lemma \ref{lem3.1} that for  $j=1, \cdots, N$,
\begin{equation}\label{4.21}
	\begin{split}
		&\big(\sqrt{-\Delta+m^{2}\epsilon_{\kappa_{n}}^{2}}-m\epsilon_{\kappa_{n}}-\epsilon_{\kappa_{n}} \nu_j^{\kappa_{n}}\big)w_{j}^{\kappa_{n}}\\
		&=\, \kappa^{\rm HF}_N\int_{\R^{3}}\frac{\rho_{\gamma^{*}_{\kappa_{n}}}(y)}{|x-y|}dy\, w_{j}^{\kappa_{n}}
		-\kappa^{\rm HF}_N\int_{\R^{3}}\frac{\gamma^{*}_{\kappa_{n}}(x,y)}{|x-y|}w_{j}^{\kappa_{n}}(y)dy =: g_{j}^{\kappa_{n}} \ \ \mbox{in} \ \, \R^3,
	\end{split}
\end{equation}
where $\epsilon_{\kappa_{n}}>0$ is as in \eqref{4.6} and satisfies $ \epsilon_{\kappa_{n}}\to0$ as $n\to\infty$, and the constant $\nu_j^{\kappa_{n}}$ satisfies 
\begin{equation}\label{4.23}
	-\infty<\lim\limits_{n\to\infty}\epsilon_{\kappa_{n}} \nu_j^{\kappa_{n}}<0,\ \ j=1,\cdots, N.
\end{equation}

The same argument of \eqref{uj-H1bdd} gives that the sequence $\big\{w_{j}^{\kappa_n}\big\}_{n\geq 1}$ is bounded uniformly in $H^1(\R^3)$ for $ j=1, \cdots,N$. Using the Hardy--Kato inequality \eqref{eq:HK} and the standard Hardy inequality $|x|^{-2}\le 4(-\Delta)$ in $\R^3$, we derive that for $j=1,\cdots, N$,
\begin{align}\label{4.21a} &\Big\|\nabla\Big[\Big(|x|^{-1}\ast\big(w_{j}^{\kappa_{n}}w_{k}^{\kappa_{n}}\big)\Big)w_{l}^{\kappa_{n}}\Big]\Big\|_{2}\nonumber\\[1mm]
\leq&2\Big\|\Big(|x|^{-1}\ast\big(w_{j}^{\kappa_{n}}w_{k}^{\kappa_{n}}\big)\Big)\nabla w_{l}^{\kappa_{n}}\Big\|_{2}
	+2\Big\|w_{l}^{\kappa_{n}}\nabla\Big(|x|^{-1}\ast\big(w_{j}^{\kappa_{n}}w_{k}^{\kappa_{n}}\big)\Big)\Big\|_{2}\nonumber\\
\leq&2\Big\||x|^{-1}\ast\big(w_{j}^{\kappa_{n}}w_{k}^{\kappa_{n}}\big)\Big\|_{\infty}\|\nabla w_{l}^{\kappa_{n}}\|_{2} +2\Big\|\nabla\Big(|x|^{-1}\ast\big(w_{j}^{\kappa_{n}}w_{k}^{\kappa_{n}}\big)\Big)\Big\|_{\infty}\|w_{l}^{\kappa_{n}}\|_{2}\nn\\[1mm]
\leq& 2\Big\||x|^{-1}\ast\big(w_{j}^{\kappa_{n}}w_{k}^{\kappa_{n}}\big)\Big\|_{\infty}\|\nabla w_{l}^{\kappa_{n}}\|_{2}+6\|w_{l}^{\kappa_{n}}\|_{2}\Big\||x|^{-2}\ast\big(w_{j}^{\kappa_{n}}w_{k}^{\kappa_{n}}\big)\Big\|_{\infty}\\[1mm]
\leq&2\pi\| w_{l}^{\kappa_{n}}\|_{H^1}\big\langle w_{j}^{\kappa_{n}}, \sqrt{-\Delta} w_{j}^{\kappa_{n}}\big\rangle^{\frac{1}{2}}\, \big\langle w_{k}^{\kappa_{n}}, \sqrt{-\Delta} w_{k}^{\kappa_{n}}\big\rangle^{\frac{1}{2}}\nn\\
& +24\|w_{l}^{\kappa_{n}}\|_{H^1}\big\langle w_{j}^{\kappa_{n}}, -\Delta w_{j}^{\kappa_{n}}\big\rangle^{\frac{1}{2}}\, \big\langle w_{k}^{\kappa_{n}}, -\Delta w_{k}^{\kappa_{n}}\big\rangle^{\frac{1}{2}}\le C \nonumber
\end{align}
uniformly for sufficiently large $n>0$. This further implies the following $H^1$-uniform estimate
\begin{align}\label{4.1}
	\limsup\limits_{n\to\infty}\int_{\R^3}|\nabla g_{j}^{\kappa_{n}}|^{2}dx<\infty,\ \ j=1, \cdots, N.
\end{align}
Note from \eqref{4.21} and \eqref{4.23} that
\begin{align*}
	&\limsup\limits_{n\to\infty}\int_{\R^3}|\xi|^4 \, \big|\big(w_j^{\kappa_n}\big)^\wedge(\xi)\big|^2d\xi\\ &=\limsup\limits_{n\to\infty}\int_{\R^3}\frac{|\xi|^4\ \big|(g_j^{\kappa_n})^\wedge(\xi)\big|^2}{\big[\sqrt{|\xi|^2+m^2\epsilon_{\kappa_n}^2}-\big(\epsilon_{\kappa_n}m+\epsilon_{\kappa_n}\nu_j^{\kappa_n}\big)\big]^2}d\xi\\[1.5mm]
	&\leq\limsup\limits_{n\to\infty}\int_{\R^3}|\xi|^2\big|(g_j^{\kappa_n})^\wedge(\xi)\big|^2d\xi\\[1.5mm]
	&=\limsup\limits_{n\to\infty}\int_{\R^3}|\nabla g_j^{\kappa_n}|^2dx<\infty,\ \ j=1,\cdots,N,
\end{align*}
where the last inequality follows from  \eqref{4.1}.
Thus, $w_{j}^{\kappa_{n}}$ is bounded uniformly in $H^2(\R^3)$ as $n\to \infty$. Moreover, since $w_{j}^{\kappa_{n}}$ satisfies the uniform decay \eqref{4.14} for sufficiently large $n>0$, it yields from Sobolev's embedding theorem that
\begin{equation*}
	w_{j}^{\kappa_{n}}\to w_{j}\ \ \, \text{strongly\ in}\, \ L^\infty(\R^3)\ \ \mathrm{as}\ \ n\rightarrow\infty.
\end{equation*}
The proof of Theorem \ref{th2} is therefore complete.
\end{proof}

\bigskip
\noindent{\bf Data availability.} Data sharing not applicable to this article as no datasets were generated or
analyzed during the current study.



\linespread{1}

\end{document}